\documentclass[11pts]{amsart}

\usepackage{amsmath,amsthm,amscd,amssymb,bbm,color}
\usepackage{latexsym}
\usepackage{fullpage}
\usepackage{tikz-cd}
\usepackage{mathtools}
\usepackage[toc]{appendix}

\usepackage{ifthen}

\usepackage{dsfont}

\usepackage{setspace}
\usepackage{nohyperref}

\usepackage [english]{babel}
\usepackage [autostyle, english = american]{csquotes}
\MakeOuterQuote{"}

\numberwithin{equation}{section}

\newtheorem{lem}{Lemma}[section]
\newtheorem{prop}{Proposition}[section]
\newtheorem{thm}{Theorem}[section]

\newtheorem{conj}{Conjecture}[section]

\newtheorem{rem}{Remark}[section]

\usepackage{thmtools,thm-restate}

\DeclareMathOperator{\C}{\mathbb{C}}
\DeclareMathOperator{\E}{\mathbb{E}}
\DeclareMathOperator{\R}{\mathbb{R}}
\DeclareMathOperator{\Q}{\mathbb{Q}}
\DeclareMathOperator{\Z}{\mathbb{Z}}
\DeclareMathOperator{\N}{\mathbb{N}}
\DeclareMathOperator{\F}{\mathbb{F}}

\DeclareMathOperator{\e}{e}
\DeclareMathOperator{\one}{\mathds{1}}

\usepackage{dsfont}

\let\oldsquare\square
\renewcommand{\square}[1][-1pt]{%
 {\raisebox{#1}{$\oldsquare$}}%
}

\newcommand{\Mod}[1]{\ (\text{mod}\ #1)}

\title{Random Chowla's Conjecture for Rademacher Multiplicative Functions}
\author{Jake Chinis and Besfort Shala} 

\begin{document}

\begin{abstract}
We study the distribution of partial sums of Rademacher random multiplicative functions $(f(n))_n$ evaluated at polynomial arguments. We show that for a polynomial $P\in \Z[x]$ that is a product of at least two distinct linear factors or an irreducible quadratic satisfying a natural condition, there exists a constant $\kappa_P>0$ such that
    \[
    \frac{1}{\sqrt{\kappa_P N}}\sum_{n\leq N}f(P(n))\xrightarrow{d}\mathcal{N}(0,1),
    \]
    as $N\rightarrow\infty$, where convergence is in distribution to a standard (real) Gaussian. This confirms a conjecture of Najnudel and addresses a question of Klurman-Shkredov-Xu.

We also study large fluctuations of $\sum_{n\leq N}f(n^2+1)$ and show that there almost surely exist arbitrarily large values of $N$ such that
\[
\Big|\sum_{n\leq N}f(n^2+1)\Big|\gg \sqrt{N \log\log N}.
\]
This matches the bound one expects from the law of iterated logarithm.
\end{abstract}

\maketitle

\section{Introduction}

The study of mean values of multiplicative functions has played a central role in analytic number theory for hundreds of years. 
At the heart of the subject lies the M{\"o}bius function, $\mu$, which is the multiplicative function supported on squarefree (natural) numbers and defined to be $-1$ on the primes. The associated generating series is $\sum_n\mu(n)/n^s$, 
which is absolutely convergent for $\Re(s)>1$, and is equal to $1/\zeta(s)$, where $\zeta$ denotes the Riemann zeta-function. It is this connection with the Riemann zeta-function which makes the study of partial sums of $\mu$ so intriguing; indeed, 
the Riemann Hypothesis is equivalent to the estimate
\begin{align}
\label{RH_sqrt}
\sum_{n\leq x}\mu(n)\ll_\varepsilon x^{1/2+\varepsilon},
\end{align}
for any $\varepsilon>0$, and all $x\geq 2$. The above estimate is often framed in terms of the "pseudo-randomness" of the M{\"o}bius function (e.g., \cite[p. 338]{IwaKow}). This paper aims to further investigate this pseudo-random behaviour through the study of so-called \textit{random multiplicative functions}.


\subsection{Random Multiplicative Functions}

A naive heuristic which one can use in order to obtain an intuitive understanding of why (\ref{RH_sqrt}) should be true is to model the M{\"o}bius function by a sequence of i.i.d.\;random variables $(X_n)_{\text{$n$:SF}}$ (indexed by the set of squarefree integers) and taking the values $\pm 1$ with equal probability. If this were the case, then the partial sums $\sum_{n\leq x}\mu(n)$ would mimic a random walk with mean zero and variance equal to number of squarefree integers up to $x$ (which is $\sim \frac{x}{\zeta(2)}$); in particular, we would expect that these partial sums fall within one standard deviation with high probability, so that (\ref{RH_sqrt}) should hold. 

In \cite{Levy}, L{\'e}vy objects to the above model, as the sequence of random variables $(X_n)_{\text{$n$:SF}}$ lacks the multiplicative structure inherent in $(\mu(n))_n$. In an attempt to rectify this, Wintner \cite{Wintner} introduces the concept of {random multiplicative functions} (RMFs), which we now describe.

Let $(f(p))_p$ denote a sequence of i.i.d.\;random variables indexed by the primes and taking the values $\pm 1$ with equal probability. A \textit{Rademacher} random multiplicative function is a sequence of random variables $(f(n))_n$ defined multiplicatively by
\begin{align*}
f(n)
:=
\begin{cases} 
      \prod_{p|n}f(p) & \text{if $n$ is squarefree,} \\
      0 & \text{otherwise}.
   \end{cases}
\end{align*}
Note that $f(n)$ is a random model for $\mu(n),$ as $\mu(n)$ is simply a specific realization of $f(n)$. Using the theory of Dirichlet series, Wintner \cite{Wintner} showed that $\sum_n f(n)n^{-s}$ is almost always convergent for $\Re(s)>1/2.$ He further showed that, for all $\varepsilon>0$, both $\sum_{n\leq x} f(n)= O_\varepsilon( x^{1/2+\varepsilon})$ and $\sum_{n\leq x} f(n)\neq O_\varepsilon (x^{1/2-\epsilon})$ hold almost always, leading to the egregious statement that "the Riemann Hypothesis is almost surely true".

The study of RMFs has flourished in recent years, most notably with work of Harper (e.g., \cite{Harper_RMF_1,Harper_RMF_2,Harper_Fluctuations, Harper_Dirichlet_SRC,Harper-Maks_Helson}). In particular, there are instances where one is able to bridge the gap between this idealized probabilistic realm and the deterministic setting, at least in the case of \textit{Steinhaus} RMFs. 

A \emph{Steinhaus} random multiplicative function is defined in a similar way as a Rademacher RMF, but with $f(p)$ distributed uniformly on the unit circle and with $f(n)$ being completely multiplicative. For instance, these RMFs are meant to model the Archimedian characters $n^{it}$ ($t$ real). In \cite[Theorem 2]{Harper_Dirichlet_SRC}, Harper shows that low moments of the Dirichlet polynomials $\sum_{n\leq N} n^{it}$ show better than squareroot cancellation on average over $t$, which is a corollary of the analogous statement for Steinhaus RMFs (via the so-called "Bohr correspondence").

As mentioned earlier, the purpose of this paper is to further investigate the relationship between random multiplicative functions and number-theoretic questions of interest; in particular, we will study the RMF analogue of the celebrated Chowla conjecture.

\subsection{The Random Chowla Conjecture}

Chowla's conjecture concerns the autocorrelations of the M\"obius function among linear forms.

\begin{conj}[Chowla's Conjecture \cite{ChowlaConj}]
    For any positive integers $a_1, b_1, \ldots, a_k, b_k$ such that $a_ib_j - a_jb_i\neq 0$, for all $i\neq j$, we have
    \[
    \sum_{n\leq N}\mu(a_1n + b_1)\cdots\mu(a_kn+b_k)=o(N),
    \]
    as $N\rightarrow\infty$.
\end{conj}

Although Chowla's conjecture is far from being resolved, there has been much progress in recent years (e.g., \cite{Log_Chowla_1,Log_Chowla_3,Log_Chowla_2,MatRadTao_AvgChowla,Helfgott-Maks_LogChowla,Cedric_TwoPoint}). See also the related Elliott's conjecture \cite{MatRadTao_AvgChowla, klurman2023elliottsconjectureapplications} for more general multiplicative functions. Much like Wintner, we wish to study these conjectures through the lens of random multiplicative functions.

For Steinhaus RMFs, the probabilistic analogue for the two-point Chowla conjecture was first studied by Najnudel \cite{Najnudel_Chowla}. He conjectured that the normalized partial sums $\frac{1}{\sqrt{N}}\sum_{n\leq N}f(n(n+1))$ converge in distribution to a standard complex Gaussian, as $N\rightarrow \infty$. This was proved and generalized by Klurman-Shkredov-Xu \cite{Oleksiy-RandomChowla} for a larger class of polynomial arguments: they showed that the normalized partial sums $\frac{1}{\sqrt{N}}\sum_{n\leq N}f(P(n))$ converge in distribution to a standard complex Gaussian, for any admissible $P\in \Z[x]$.

\begin{thm}[Steinhaus Random Chowla \cite{Oleksiy-RandomChowla}]
\label{Steinhaus-RC}
    Let $f$ be a Steinhaus random multiplicative function. Then for any polynomial $P\in\Z[x]$ of $\deg P \geq 2$ which is not of the form $P(x) = w(x + c)^d$ for some $w \in \Z, c \in \Q$, we have that
    \[
    \frac{1}{\sqrt{N}}\sum_{n\leq N}f(P(n)) \xrightarrow{d} \mathcal{CN}(0,1),
    \]
    as $N\rightarrow\infty$; that is, the normalized partial sums $\frac{1}{\sqrt{N}}\sum_{n\leq N}f(P(n))$ converge in distribution to a standard complex Gaussian.
\end{thm}

The main technical difficulty in proving the above theorem lies in computing the fourth moment for Steinhaus RMFs evaluated at polynomial arguments, which amounts to counting the number of $m_1,m_2,n_1,n_2\leq N$ for which $P(m_1)P(m_2)=P(n_1)P(n_2)$.  
Klurman-Shkredov-Xu \cite[Theorem 3.2]{Oleksiy-RandomChowla} view this counting problem through the lens of "multiplicative energy", 
and, using strong results from the theory of integral points on absolutely irreducible curves, ultimately prove the correct asymptotic. It turns out that one can compute all finite moments for Steinhaus RMFs at polynomial arguments; this was accomplished by Wang-Xu in \cite{Wang-Xu}. 

Inspired by the techniques used by Wang-Xu in \cite{Wang-Xu}, we are able to prove the more delicate analogue of Theorem \ref{Steinhaus-RC} for Rademacher RMFs. This confirms a conjecture of Najnudel \cite{Najnudel_Chowla} in the Rademacher case 
and addresses the question in \cite[Section 1.4]{Oleksiy-RandomChowla}.

\begin{thm}[Rademacher Random Chowla]
    \label{main_thm}
    Let $f$ be a Rademacher RMF and let $P\in \Z[x]$ be a product of at least two distinct linear factors over $\Z$ or irreducible of degree $2$. Suppose that there is no prime $p$ for which $p^2 \mid P(n)$ for all $n\in \N$. Then, there exists a constant $\kappa_P>0$ such that
    \[
    \frac{1}{\sqrt{\kappa_P N}}\sum_{n\leq N}f(P(n))\xrightarrow{d}\mathcal{N}(0,1),
    \]
    as $N\rightarrow\infty$; that is, the partial sums $\sum_{n\leq N}f(P(n)),$ when appropriately normalized, converge in distribution to a standard (real) Gaussian, as $N\rightarrow\infty.$
\end{thm}

\begin{rem}
    If there is no prime $p$ such that $p^2\mid P(n)$ for all $n\in \N$, then we call the polynomial $P$ \textbf{admissible}.
\end{rem}

The proofs of Theorems \ref{Steinhaus-RC} and \ref{main_thm} start off with the same strategy used by Harper in \cite{Harper_Martingale}, where he exploits the martingale difference sequence structure provided in the partial sums of RMFs: proving a Central Limit Theorem for a martingale difference sequence amounts to understanding the second and fourth moment for the random variables in question (McLeish's CLT \cite{McLeish} or Lemma \ref{McLeish_CLT})\footnote{We remark that Soundararajan and Xu \cite{Xu_CentralLimitTheorems} have studied the more general case for when the partial sums $\sum_{n\leq N}a_nf(n)$ satisfy a central limit theorem, for deterministic coefficients $(a_n)_n$.}. In the case of Steinhaus RMFs, the second and fourth moment estimates amount to obtaining asymptotics for the following:
\begin{align*}
\#\{m,n\leq N: P(m)=P(n)\}\;\;\&\;\;\#\{m_1,m_2,n_1,n_2\leq N: P(m_1)P(m_2)=P(n_1)P(n_2)\},
\end{align*}
respectively. The former is trivial, as $P$ will be injective for all sufficiently large arguments. The latter was dealt with initially in \cite{Oleksiy-RandomChowla} and later in \cite{Wang-Xu} as a special case (with a related but different proof).

The situation for Rademacher RMFs is more complicated. Indeed, the second and fourth moment estimates in our setting amount to obtaining asymptotics for
\begin{align}
\#\{n_1,n_2\leq N:P(n_1)P(n_2)=\square,\text{  $P(n_i)$ squarefree}\}
\end{align}
and
\begin{align}
\label{counting_problem}
\#\{n_1,n_2,n_3,n_4\leq N: P(n_1)P(n_2)P(n_3)P(n_4)=\square,\text{  $P(n_i)$ squarefree}\},
\end{align}
respectively. The former is still ``trivial'', as a product of two squarefree numbers is a perfect square if and only if the two numbers are equal; and so, the second moment amounts to counting the number of $n\leq N$ for which $P(n)$ is squarefree. This is a notoriously difficult problem for general polynomials, which we briefly discuss at the end of Section \ref{preliminaries}. The main technical result in this paper is proving an asymptotic for \eqref{counting_problem} (Proposition \ref{counting_squares}). The key new ingredient to our proof is a bootstrapping argument, where it is essential that our polynomials are either products of linear factors or irreducible quadratics.

\subsection{Large Fluctuations} 

After establishing a central limit theorem for the partial sums $\sum_{n\leq N} f(P(n))$, it becomes natural to investigate their large fluctuations. Given a sequence of i.i.d.\;random variables $(X_n)_n$, with mean $0$ and variance $1$, the law of iterated logarithm asserts that
\[
\limsup_{N\to\infty} \frac{\lvert \sum_{n\leq N} X_n\rvert}{\sqrt{2N\log\log N}} = 1,
\]
almost surely. As we shall presently see, we show that this type of lower bound also holds for $\sum_{n\leq N} f(n^2 + 1)$. 

Large fluctuations of $\sum_{n\leq N} f(n)$ where $f$ is a Rademacher or Steinahus RMF have been addressed in \cite{Harper_Fluctuations} -- see the introduction for a historical overview. Large fluctuations of Steinhaus RMFs evaluated at polynomial arguments were studied in \cite{Oleksiy-RandomChowla}. Building upon \cite{Harper_Fluctuations} and following the strategy in \cite{Oleksiy-RandomChowla}, we are able to prove the corresponding lower bound in the law of iterated logarithm\footnote{One should view $n^2 + 1$ here as representative of irreducible quadratic polynomials, although there is a minor caveat with this which we explain in Remark \ref{rem_quadr}.} for $\sum_{n\leq N} f(n^2 + 1)$. The corresponding upper bound and the more delicate case when $P$ is a product of linear factors (for Steinhaus as well as Rademacher RMFs) will be addressed in future work. 

\begin{thm}
    \label{large_values_theorem} Let $f$ be a Rademacher $RMF$. There exist arbitrarily large values of $N$ such that
    \begin{equation}
        \Big|\sum_{n\leq N} f(n^2 + 1)\Big| \gg \sqrt{N\log\log N},
    \end{equation}
    almost surely.
\end{thm}

For a detailed outline of the proof, see the beginning of Section \ref{large_values_section}. 

\begin{rem}
One could ask whether Theorem \ref{large_values_theorem} reflects the truth of large fluctuations in the deterministic case of partial sums of the M\"obius function. This is a difficult question in general, although Gonek's conjecture and refinements due to Ng \cite{Ng_2004} predict the existence of a positive constant $c$ such that 
\begin{equation}
    \limsup_{N\to\infty} \frac{\sum_{n\leq N}\mu(n)}{\sqrt{N}(\log\log\log N)^{5/4}} = c.
\end{equation}
The quantity $\sqrt{N}(\log\log\log N)^{5/4}$ is much smaller than one would naively conjecture based on the law of iterated logarithm. Nonetheless, when evaluating $\mu$ on admissible polynomials, we believe that the random prediction does capture the truth of large fluctuations. In particular, we share the belief with the authors of \cite{Oleksiy-RandomChowla} that $\sum_{n\leq N} \mu(n^2 + 1) = O(\sqrt{N\log\log N})$ holds and that this is sharp. 
\end{rem}

\subsection{Completely Multiplcative RMFs and Quadratic Twists}

We close the introduction section with some remarks on \textit{completely multiplicative} Rademacher RMFs. Instead of modeling the M{\"o}bius function by a random variable, one can choose to model the Liouville function instead. The difference here is that the corresponding RMF should be completely multiplicative, taking the values $\pm 1$ on the primes with equal probability. We continue to have the desired martingale structure, so that Lemma \ref{McLeish_CLT} continues to apply, but now the second and fourth moment estimates correspond to obtaining asymptotics for the following:
\begin{align*}
\#\{n_1,n_2\leq N: P(n_1)P(n_2)=\square\}\;\;\&\;\;\#\{n_1,n_2,n_3,n_4\leq N: P(n_1)P(n_2)P(n_3)P(n_4)=\square\},
\end{align*}
respectively, where we emphasize that the restriction to squarefree polynomial values has been dropped. Already for the second moment in this setting, there are off diagonal terms which require some work to count. Indeed, given $P(n_i)$, this can be written uniquely as $P(n_i)=d_iy_i^2$, for some squarefree $d_i$; in particular, if $P(n_1)P(n_2)=\square,$ then necessarily $d_1=d_2$. Hence, the second moment estimate becomes
\begin{align*}
\#\{n_1,n_2\leq N: P(n_1)P(n_2)=\square\}
=
\sum_d \left(\sum_{x\leq N}\one_{P(x)=d\square}\right)^2,
\end{align*}
which counts the number of integral points up to a given height on the family of quadratic twists of the curve with affine model $C:y^2=P(x)$. In the case where $P$ has degree $2$, we are able to solve this counting problem using work of Hooley on Pell equations \cite{Hooley-Pell}: Theorem 1 of \cite{Hooley-Pell} tells us that almost all twists will have at most one integral point of low height, so that we can replace the second moment for these $d$ by the first moment, which exactly counts diagonal solutions $n_1=n_2$. For the remaining twists, trivial bounds suffice. For higher degree polynomials, analogues of Hooley' results exist, conditionally on the $abc$-conjecture and for suitable $P$, by work of Granville \cite[Theorem 4 (i)]{Granville-twists}; as such, we are able to solve special cases of the second moment problem for completely multiplicative Rademacher RMFs. The fourth moment is substantially more complicated. For starters, one would like to know that almost all twists of the surface $y^2=P(x_1)P(x_2)$ have at most one integral point. Nevertheless, we still expect the behaviour to be Gaussian and this is a subject of ongoing investigation.

\subsection{Outline of the paper} 

Our paper is organized as follows. In Section \ref{preliminaries}, we provide the necessary probabilistic background needed to understand the hypotheses to McLeish's CLT (Lemma \ref{McLeish_CLT}). Section \ref{preliminaries} also contains standard results on polynomial congruences, integral points on absolutely irreducible curves, and squarefree polynomial values. In Section \ref{main_section}, we provide a proof of our main technical result (Proposition \ref{counting_squares}), which effectively solves the counting problem discussed in (\ref{counting_problem}). In the penultimate section of this paper, we incorporate the martingale difference sequence structure with the counting lemmas we have proven, thereby verifying the hypotheses to McLeish's CLT and proving Theorem \ref{main_thm}. Finally, in Section \ref{large_values_section}, we prove Theorem \ref{large_values_theorem} on large values of $\frac{1}{\sqrt{N}}\sum_{n\leq N}f(n^2+1)$.

\section{Preliminaries and Initial Reductions}
\label{preliminaries}

In this section, we collect several preliminary results which we use throughout our paper. We begin with the necessary probabilistic background on martingale difference sequences and McLeish's Central Limit Theorem for Martingales (Lemma \ref{McLeish_CLT}). We remark that Lemma \ref{McLeish_CLT} is the starting point for much of the literature surrounding RMFs (e.g., \cite{Harper_Martingale,Oleksiy-RandomChowla}).

\subsection{Martingales and McLeish's CLT}

Let $(\Omega, \mathcal{F}, \mathbb{P})$ be a probability space and let $(\mathcal{F}_n)_n$ be a filtration on $(\Omega, \mathcal{F}, \mathbb{P})$ (i.e., a sequence of sub-$\sigma$-algebras satisfying $\mathcal{F}_0\subseteq \mathcal{F}_1\subseteq\dots\subseteq \mathcal{F}$). A sequence of random variables $(X_n)_n$ is called a \textit{martingale difference sequence} on $(\Omega, \mathcal{F}, (\mathcal{F}_n)_n,\mathbb{P})$ if
\begin{enumerate}
    \item $\E\left[|X_n|\right]<\infty$ for all $n\in\N;$
    \item $\E\left[X_n|\mathcal{F}_{n-1}\right]=0$ a.s.\;for all $n\in \N$.
\end{enumerate}

The following result is a Central Limit Theorem for multi-indexed martingales:

\begin{lem}[McLeish's CLT - {\cite[{Corrolary 2.13}]{McLeish}}]
\label{McLeish_CLT}
Let $(k_N)_N\subset \N$ and suppose $(X_{i,N})_i$ is a \textit{martingale difference sequence} on $(\Omega,\mathcal{F},(\mathcal{F}_{i,N})_i,\mathbb{P})$ for $1\leq i\leq k_N$. Set $S_N:=\sum_{i\leq k_N}X_{i,N}$ and suppose the following conditions hold:
\begin{enumerate}
    \item (Normalized Variances) $\sum_{i\leq k_N}\E\left[X_{i,N}^2\right]\rightarrow 1,$ as $N\rightarrow \infty$.
    \item (Lindeberg Condition) For all $\varepsilon>0$, $\sum_{i\leq k_N}\E\left[X_{i,N}^2\one_{|X_{i,N}|>\varepsilon}\right]\rightarrow 0,$ as $N\rightarrow \infty$.
    \item (Cross-terms Condition) $\limsup_{N\rightarrow \infty} \left(\sum_{i,j\leq k_N, i\neq j}\E\left[X_{i,N}^2X_{j,N}^2\right]\right)\leq 1$.
\end{enumerate}
Then: 
$$S_N\xrightarrow{d}\mathcal{N}(0,1),$$
as $N\rightarrow\infty;$ that is, $S_N$ converges in distribution to a normal random variable with mean $0$ and variance $1$.

\end{lem}

\begin{rem}
\label{fourth_moment_condition}
    We remark that condition (2) can be replaced (via Chebyshev's Inequality) by the stronger: 
    \begin{enumerate}
      \setcounter{enumi}{1}
        \item (Fourth Moment Condition) $\sum_{i\leq k_N}\E\left[X_{i,N}^4\right]\rightarrow 0,$ as $N\rightarrow \infty$.
    \end{enumerate}
\end{rem}

Let 
\[
M_N:=\sum_{n\leq N}f(P(n))\;\;\&\;\;M_{p,N}:=\sum_{\substack{{n\leq N}\\{P^+(P(n))=p}}}f(P(n))
\]
and consider the normalized partial sums
\[
S_N:=\frac{1}{\sqrt{\E\left[M_N^2\right]}}M_N\;\;\&\;\;S_{p,N}:=\frac{1}{\sqrt{\E\left[M_N^2\right]}}M_{p,N}.
\]
We apply Lemma \ref{McLeish_CLT} with $k_N:=\max\{p:p\mid P(j),j\leq N\}$ and with $(X_{i,N})_i$ a sequence of random variables indexed by the primes:
\[
X_{i,N} 
\equiv
X_{p,N}
:=
S_{p,N},
\]
so that
\[
S_N
=
\frac{1}{\sqrt{\E\left[M_N^2\right]}}\sum_{n\leq N}f(P(n))
=
\sum_{p\leq k_N}\left(\frac{1}{\sqrt{\E\left[M_N^2\right]}}\sum_{\substack{{n\leq N}\\{P^+(P(n))=p}}}f(P(n))\right)
=
\sum_{p\leq k_N} S_{p,N}.
\]
We remark that $(S_{p,N})_p$ is a martingale difference sequence for all $N$, relative to the natural filtration
\begin{align*}
\mathcal{F}_{p,N}
&:=\sigma\left(\{S_{q,N}:q\leq p\right\})\\
&=
\text{the sigma algebra generated by the set $\{S_{q,N}:q\leq p\}$}.
\end{align*}
Thus, the bulk of this paper 
is dedicated to evaluating $\E\left[M_N^2\right]$ and to show that the sequence of random variables $(S_{p,N})_p$ satisfies hypotheses (1), (2), and (3) of Lemma \ref{McLeish_CLT}. In order to do so, we first need some results on squarefree values of, and general divisor bounds for, integer polynomials, as well as results on integral points of curves. 

\subsection{Squarefree Values and Divisor Bounds}

In Section \ref{Section_Proof}, we use Lemma \ref{McLeish_CLT} to reduce the proof of Theorem \ref{main_thm} to computing the second and fourth moment of $S_{N}$. We make use of the following results on polynomial congruences and of the celebrated theorem of Bombieri-Pila on integral points of absolutely irreducible curves.

\begin{lem}[Lagrange's Theorem]
\label{Lagrange_Prime}
Let $Q\in\Z[x]$, $y,N\in\N$, and $\omega(y):=\#\{p:p\mid y\}$. If $y$ is squarefree and $Q$ is not the zero polynomial modulo $p$ for any prime $p\mid y$, then:
    \begin{enumerate}
        \item $$\#\{x\Mod{y}:  Q(x)\equiv 0\Mod{y}\}\leq (\deg Q)^{\omega(y)};$$
        \item $$\#\{x\leq N:  Q(x)\equiv 0\Mod{y}\}\leq (\deg Q)^{\omega(y)}\left(\frac{N}{y}+1\right).$$
    \end{enumerate}
\end{lem}
\begin{proof}
    The first part follows from Lagrange's Theorem on solutions of polynomial congruences with prime moduli (e.g.,\;Theorem 5.21 of \cite{Apo}), together with the Chinese Remainder Theorem (e.g.,\;Theorem 5.28 of \cite{Apo}). The second part trivially follows from the first, by splitting the interval $[1,N]$ into $\leq N/y+1$ complete residue systems.
\end{proof}

Among much of the case work that is required to prove Theorem \ref{main_thm}, some cases reduce to counting integral points on curves. As such, we will make use of the celebrated theorem of Bombieri-Pila and a less general result which is better suited for curves of degree $2$, explicitly stated by Cilleruelo and Garaev in \cite{Jacques_Quaratic}; we thank Jacques Benatar for providing the latter reference.

\begin{lem}[Bombieri-Pila {\cite[Theorem 5]{Bombieri-Pila}}]
\label{Bombieri-Pila_Lemma}
    Let $\mathcal{C}$ be an absolutely irreducible curve of degree $d\geq 2$ and let $N$ be a positive integer $\geq \exp(d^6)$. Then the number of integral points on $\mathcal{C}$ and inside the square $[0,N]\times [0,N]$ is at most $N^{1/d}\exp\left(12\sqrt{d\log N\log\log N}\right).$
\end{lem}

\begin{lem}[Integral Points on Quadratic Curves {\cite[Proposition 1]{Jacques_Quaratic}}]
\label{Quadratic_Curves}
Let $|A|,|B|,|C|,|D|,|E|,|F|\leq N^{O(1)}$ and assume that $B^2 - 4AC\neq \square$. Then, the Diophantine equation
\[
Ax^2 + Bxy + Cy^2 + Dx + Ey + F = 0 
\]
has at most $N^{o(1)}$ integral solutions with $1 \leq |x|, |y| \leq N^{O(1)}$.
\end{lem}

We apply Lemma \ref{Bombieri-Pila_Lemma} to curves given by the zero set of $F(x,y):=aP(x)-bP(y)$, where $a\neq b$ with $a,b\ll N^{O_{\deg P}(1)}$. Verifying that $F(x,y)=0$ is an absolutely irreducible curve for all such $a,b$ is a non-trivial task (and false in some cases); as such, we aim to show that all quadratic factors of $F(x,y)$ will be of the form in Lemma \ref{Quadratic_Curves}. This, together with Lemma \ref{Bombieri-Pila_Lemma}, will give us the appropriate saving we need to count integral points on $F(x,y)=0$ uniformly in $a,b\ll N^{O_{\deg P}(1)}$. 

\begin{lem}[Integral Points on Curves]
\label{New_BP}
Let $P\in\Z[x]$ be a polynomial with no repeated roots, let $\varepsilon>0$, and suppose $\deg P\geq 2$. Moreover assume that the coefficients of $P$ are $\ll N^{O_{\deg P}(1)}$. Then, the curve $F(x,y)=0$, defined by $F(x,y):=aP(x)-bP(y)\in\Z[x,y],$ has at most 
\[
\ll_{\varepsilon} N^{1/3+\varepsilon}
\]
integral points with $1\leq x,y\leq N$, uniformly over all such $P$ and $a,b\ll N^{O_{\deg P}(1)},$ $a\neq b$. 
\end{lem}
\begin{proof}
    Given $F(x,y),$ consider its prime factorization $f_1(x,y)\cdots f_k(x,y)$ in $\C[x,y]$ and note that $F(x,y)=0$ if and only if $f_i(x,y)=0$ for some $i$; in particular, the number of integral points on $F(x,y)=0$ is the sum of the number of integral points on each $f_i(x,y)$. We note that none of the $f_i$'s have degree $1$; see e.g., \cite[Lemma 2.5]{Wang-Xu}. 
    
    For $f_i$ with $\deg f_i\geq 3$, there are at most $\ll_{\varepsilon}N^{1/\deg f_i + \varepsilon}\ll N^{1/3+\varepsilon}$ integral points in the desired box, by Lemma \ref{Bombieri-Pila_Lemma}. 
    
    For $f_i$ of degree $2$, we may assume, without loss of generality, that $f_i(x,y)\in\Q[x,y]$: if not, we note that $f_i$ comes with a conjugate pair\footnote{Since $F(x,y)\in\Z[x,y],$ if $f_i$ has degree $2$ and $f_i\not\in\Q[x,y]$, then there must be an $f_j$ such that $f_if_j\in\Q[x,y].$}, say $f_j$, as $F(x,y)\in\Z[x,y];$ in particular, the number of rational points on $f_i(x,y)=0$ is equal to the number of rational points on the intersection of the curves $f_i(x,y)=0=f_j(x,y)$, which is at most $4$ by Bezout's Theorem on intersections of plane curves \cite{Bezout_Intersections}. In the case where $F(x,y)$ has quadratic factors in $\Q[x,y]$, then the coefficients of this quadratic factor are $\ll N^{O_{\deg P}(1)}$ whenever $a,b\ll N^{O_{\deg P}(1)}$. To see this, note that the left hand side of $F(x,y)=f_1(x,y)\cdots f_k(x,y)$ is at most $N^{O_{\deg P}(1)}$ for $a,b\ll N^{O_{\deg P}(1)}$, $x,y\leq N$; comparing coefficients yields the desired result. Alternatively, one can iterate work of Granville on bounds of coefficients of factors of single variable polynomials \cite{Granville_FactorBounds}. If the discriminant of the quadratic factor is not a square, we apply Lemma \ref{Quadratic_Curves} and get that there are at most $\ll_{\varepsilon} N^\varepsilon$ integral points across such quadratic irreducible factors. In the case where it is a square, then the quadratic terms can be written as a product of two linear factors plus a constant term, say $\ell_1(x,y)\ell_2(x,y)+g$, with $g$ necessarily non-zero; and so, there are at most $\ll_\varepsilon N^\varepsilon$ integral points on $\ell_1(x,y)\ell_2(x,y)+g=0$ by the divisor bound and we are done.
\end{proof}

The final ingredient we need is a result on the density of squarefree values of integer polynomials. We remark that is not known whether a general polynomial with integer coefficients takes squarefree values infinitely-often, although this is known for quadratic polynomials \cite{Ricci_Quadratic_SF} and for polynomials which are a product of (distinct) linear factors. We remark that the latter result does not seem to be explicitly stated in the literature (although it follows from \cite{Klurman_2017}), so we provide an explicit statement and proof here, for future reference. For more general polynomials, we have conditional results by work of Granville \cite{Granville-twists}, which assumes the $abc$-conjecture.

\begin{lem}[Density of Squarefree Polynomial Values]
    \label{Density_SF}
    Let $P\in\Z[x]$ be an irreducible quadratic polynomial or a polynomial of arbitrary degree which factors into distinct linear factors over $\Z$. If there is no prime $p$ for which $p^2\mid P(n)$ for all $n\in\N$, then there exists a constant\footnote{The constant $\kappa_P$ can of course be computed explicitly; indeed, it is given by an Euler product, which one can also compute heuristically, using the usual sieve-theoretic/probabilistic arguments.} $\kappa_P>0$ such that
\begin{align}
\label{SF_Count_eq}
\#\{n\leq N: \text{$P(n)$ is squarefree}\}
=
\kappa_P N + O_P\left(N^{3/4}\right),
\end{align} 
as $N\rightarrow\infty$.
\end{lem}
\begin{proof}
    For the quadratic case, see either \cite{Ricci_Quadratic_SF} or the expository note of Rudnick \cite{Rudnick_Quadratic_SF}. We note that the error term for the quadratic case is actually $\ll_P N^{2/3}\log N$, but we will not make use of this extra saving.
    
    For the case where $P$ factors into a product of linear factors, the result follows almost immediately from work of Mirsky \cite{10.1093/qmath/os-18.1.178} or Tsang \cite[Theorem 3]{Tsang_HL} on correlations of squarefree integers. See also the thesis of Mennema \cite{Mennema_HL}, which improves on the error and implied constants in \cite{Tsang_HL}. As the result we are looking for does not seem to be explicitly available in the literature, we provide the details below.
    
    Suppose $\ell_1,\dots,\ell_k$ are arbitrary linear forms in $\Z[x]$, such that there are no primes $p$ for which $p^2\mid \ell_i(n)$ for all $n\in\N$ and for any $i=1,\dots,k$. Theorem 3 of \cite{Tsang_HL} tells us that
    \begin{align}
    \label{Tsang}
    \sum_{n\leq N}\mu^2(\ell_1(n))\cdots\mu^2(\ell_k(n)) = c_{\{\ell_1,\dots,\ell_k\}}N + O_{\{\ell_1,\dots,\ell_k\}}\left(N^{3/4}\right),
    \end{align}
    for some (explicit) constant $c_{\{\ell_1,\dots,\ell_k\}}>0$ and where the implied constant in the error term depends on both the degree and the coefficients of the polynomial $\prod_{i=1}^k \ell_i$. Next, let $q$ be the squarefree part of $\prod_{i\neq j}\prod_{n}\gcd(\ell_i(n),\ell_j(n))$ and note that $q$ is finite, for if $p\mid \ell_i(n),\ell_j(n)$, with $\ell_i(x)=a_ix+b_i$ say, then $p$ necessarily divides\footnote{We have that $a_ib_j-a_jb_i\neq 0$, as $P$ is such that there is no prime $p$ for which $p^2\mid P(n)$ for all $n\in\N$.} $a_ib_j-a_jb_i(\neq 0)$.  Thus, if $P(x)=(a_1x+b_1)\cdots (a_kx+b_k)\in\Z[x]$ is admissible, then
    \begin{align*}
        \#\{n\leq N: \text{$P(n)$ is squarefree}\}
        &=
        \sum_{n\leq N}\mu^2(P(n))\\
        &=
        \sum_{m\Mod{q}}\sum_{\substack{{n\leq N}\\{n\equiv m\Mod{q}}}}\mu^2((a_1n+b_1)\cdots (a_kn+b_k))\\
        &=
        \sum_{\substack{{m\Mod{q}}}}^*\sum_{\substack{{n\leq N}\\{n\equiv m\Mod{q}}}}\mu^2(a_1n+b_1)\cdots \mu^2(a_kn+b_k),       
    \end{align*}
    where we have used the fact that for all $m\Mod{q}$ and all $i\neq j$, either $(\ell_i(n),\ell_j(n))=1$ for all $n\equiv m\Mod{q}$ or $(\ell_i(n),\ell_j(n))>1$ for all $n\equiv m\Mod{q}$. To see this, note that if $m$ and $n$ are such that $(\ell_i(n),\ell_j(n))>1$ with $n\equiv m\Mod{q}$, then there exists a prime $p$ such that $p\nmid q$, by definition of $q$. Furthermore,  $n\equiv -b_ia_i^{-1} \Mod{p}$, say, as at least one of $a_i,a_j$ must not be divisible by $p$, by the assumption that $P$ is admissible. Hence, for any other $n_0\equiv m \Mod{q}$, we have that $n_0\equiv m \Mod{p}$ with $m\equiv -b_ia_i^{-1}\Mod{p}$, so that $(\ell_i(n_0),\ell_j(n_0))>1$ and we are done. 
    In particular,
    \begin{align*}
        \#\{n\leq N: \text{$P(n)$ is squarefree}\}
        &=
        \sum_{m\Mod{q}}^*
        \sum_{d\leq \frac{1}{q}(N-m)}
        \mu^2(a_1(m+dq)+b_1)\cdots \mu^2(a_k(m+dq)+b_k),
    \end{align*}
    where we have written $n=m+dq$. We now apply Tsang's (\ref{Tsang}) to obtain (\ref{SF_Count_eq}) with $\kappa_P$ equal to $\sum_{m\Mod{q}}^*\frac{1}{q}c_{\{\ell_1,\dots,\ell_k\}}$ and where $c_{\{\ell_1,\dots,\ell_k\}}$ is as in (\ref{Tsang}) with $\ell_i(x):=a_iqx + (a_im+b_i)$.
\end{proof}

We now have all the tools required to prove our main technical result (Proposition \ref{counting_squares}), from which we will easily deduce Theorem \ref{main_thm}.
\section{Products of Polynomials and Searching for Squares}
\label{main_section}

As we will shortly see, Lemma \ref{McLeish_CLT} reduces Theorem \ref{main_thm} to understanding the fourth moment of the complete partial sums $M_N=\sum_{n\leq N}f(P(n))$. As we have already seen, the second moment of $M_N$ has arithmetic significance, in that $\E\left[M_N^2\right]$ counts the number of $n\leq N$ for which $P(n)$ is squarefree. The same is true for the fourth moment, namely
\begin{align*}
    \E\left[M_N^4\right]
    &=
    \sum_{n_1,n_2,n_3,n_4\leq N}\E\left[f(P(n_1))f(P(n_2))f(P(n_3))f(P(n_4))\right]\\
    &=
    \#\{n_1,n_2,n_3,n_4\leq N: P(n_1)P(n_2)P(n_3)P(n_4)=\square\text{ with $P(n_i)$ squarefree}\}.
\end{align*}
Due to the symmetry present in the problem, and the belief that the distribution of these random variables is indeed normal, we expect that the only contribution to the above comes from the trivial solutions, where $n_i=n_j$ in pairs. The goal in this section is to prove that this indeed the case.

\begin{prop}
\label{counting_squares}
    Let $\varepsilon>0$ and let $P\in\Z[x]$ be an irreducible quadratic polynomial or a polynomial which factors into distinct linear factors over $\Z$. If there is no prime $p$ for which $p^2\mid  P(n)$ for all $n\in\N$, then there exists a constant $\delta_P>0$ such that
    \begin{align}
        \sum_{n_1,n_2,n_3,n_4\leq N}\one_{P(n_1)P(n_2)P(n_3)P(n_4)=\square}^{\text{SF}}
        &=
        3\kappa_P^2 N^2 + O_{P,\varepsilon}(N^{2-\delta_P+\varepsilon}),
    \end{align}
    as $N\rightarrow \infty$, where $\kappa_P$ is given by (\ref{Density_SF}) and where $\one_{P(n_1)P(n_2)P(n_3)P(n_4)=\square}^{\text{SF}}$ is the indicator function on when the product $P(n_1)P(n_2)P(n_3)P(n_4)$ is a perfect square, with each factor $P(n_i)$ being squarefree.
\end{prop}
\begin{proof}
Note that if $n_1=n_2$, say, and $P(n_1)P(n_2)P(n_3)P(n_4)=\square$ with each $P(n_i)$ squarefree, then necessarily $n_3=n_4$. Thus, the main term is $3\kappa_P^2N^2$ by Lemma \ref{Density_SF}. It remains to show that the remaining terms contribute $o(N^2)$, with the appropriate power saving. 

We follow the general proof strategy from \cite{Wang-Xu}, where they in particular count the number of $m_1,m_2,n_1,n_2\leq N$ for which $P(m_1)P(m_2)=P(n_1)P(n_2)$. The main idea is to fix one of the variables and use either divisor bounds for polynomial congruences (when the polynomial values do not share many common prime factors), or general point counting results for curves of the form $aP(x)-bP(y)$ (when the polynomial values do share a large common divisor). Note that 
\[
    P(n_1)P(n_2)P(n_3)P(n_4)=\square
\]
may be transformed into an equation resembling the fourth moment equation coming from the Steinhaus case as follows. Given that each $P(n_i)$ is squarefree, 
we have $P(n_1)P(n_2)P(n_3)P(n_4)=\square$ if and only if $P(n_1)P(n_2)=\mu d_1^2$ and $P(n_3)P(n_4)=\mu d_2^2$, where $\mu$ is squarefree and $d_1=(P(n_1),P(n_2)),d_2=(P(n_3),P(n_4)).$ Thus
\[
 \sum_{n_1,n_2,n_3,n_4\leq N}\one_{P(n_1)P(n_2)P(n_3)P(n_4)=\square}^{\text{SF}}
 =
  \sum_{n_1,n_2,n_3,n_4\leq N}\one_{P(n_1)P(n_2)/d_1^2=P(n_3)P(n_4)/d_2^2}^{\text{SF}}.
\]

The rest of the proof is split as follows. For small $d_1,d_2$, we ignore the fact that these $d_i$'s come from GCDs of polynomials and use the fact that $d_i$ is simply a divisor of a certain polynomial value. This reduces matters to counting solutions of $aP(x)-bP(y)=0$ for ranging $a,b$, which is winning via Bombieri-Pila (Lemmas \ref{Bombieri-Pila_Lemma}, \ref{Quadratic_Curves}, \ref{New_BP}). For larger $d_1,d_2$, we use the fact that our admissible polynomials are such that they possess smaller divisors, giving a better control on the number of roots of the polynomial congruence equations coming from the correspsonding divisibility conditions. 

Care must be taken in the above case work, as to ensure that the trivial solutions are removed in each step.

To be more precise, we wish to obtain an upper bound for the following sum (with trivial solutions set to the side), which we split into two parts (corresponding to small GCDs and large GCDs, respectively), for some parameter $D$ to be optimized later:
    \begin{align}
    \label{off-diagonal sum}
         \sum_{\substack{{n_1,n_2,n_3,n_4\leq N}\\{d_1,d_2\leq D}}}\one_{P(n_1)P(n_2)P(n_3)P(n_4)=\square}^{\text{SF}}
         +
         \sum_{\substack{{n_1,n_2,n_3,n_4\leq N}\\{\min\{d_1,d_2\}>D}}}\one_{P(n_1)P(n_2)P(n_3)P(n_4)=\square}^{\text{SF}},
    \end{align}
    where emphasize that the first sum is over all $n_1,n_2,n_3,n_4\leq N$ for which $d_1=(P(n_1),P(n_2))$ and $d_2=(P(n_3),P(n_4))$ are $\leq D$ and similarly for the second sum.

\subsection{Small $d_1,d_2$}
 The total contribution to (\ref{off-diagonal sum}) for $d_i\leq D$ is bounded above by
        \begin{align}
            \sum_{d_1,d_2\leq D}
            \sum_{\substack{{n_1,n_2,n_3,n_4\leq N}\\{d_1\mid P(n_1),P(n_2)}\\{d_2\mid P(n_3),P(n_4)}}}\one_{d_2^2P(n_1)P(n_2)=d_1^2P(n_3)P(n_4)}^{\text{SF}},
        \end{align}
        where we have trivially dropped the fact that the $d_i$'s are GCDs of some polynomials and are merely using the fact that they are divisors of the appropriate polynomials. To bound the above, we parametrize the solutions to $P(x)\equiv 0\Mod{d_i}$, by noting that there are $\ll_{P,\varepsilon}N^{\varepsilon}$ such solutions with $x \leq d_i$ by the first part of Lemma \ref{Lagrange_Prime}. In particular, we have
        \begin{align}
        &\sum_{d_1,d_2\leq D}
            \sum_{\substack{{n_1,n_2,n_3,n_4\leq N}\\{d_1\mid P(n_1),P(n_2)}\\{d_2\mid P(n_3),P(n_4)}}}\one_{d_2^2P(n_1)P(n_2)=d_1^2P(n_3)P(n_4)}^{\text{SF}}\\
            &\leq
            \sum_{d_1,d_2\leq D}
            \sum_{\substack{{n_{0,1},n_{0,2}\Mod{d_1}}\\{n_{0,3},n_{0,4}\Mod{d_2}}}}
            \sum_{\substack{{t_1,t_2\leq N/d_1}\\{t_3,t_4\leq N/d_2}}}\one_{Q_1(t_1)Q_2(t_2)=Q_3(t_3)Q_4(t_4)}^{\text{SF}},
        \end{align}
        where the sum over $n_{0,i}$ is over all $n_1,n_2$ modulo $d_1$ and all $n_3,n_4$ modulo $d_2$ for which $P(n_1),P(n_2)\equiv 0\Mod{d_1}$ and $P(n_3),P(n_4)\equiv 0\Mod{d_2},$ respectively, and where
        \begin{align}
            &Q_1(t_1) = Q_{n_{0,1},d_1}(t_1):=\frac{1}{d_1}P(n_{0,1}+d_1t_1)\\
            &Q_2(t_2) = Q_{n_{0,2},d_1}(t_2):=\frac{1}{d_1}P(n_{0,2}+d_1t_2)\\
            &Q_3(t_3) = Q_{n_{0,3},d_2}(t_3):=\frac{1}{d_2}P(n_{0,3}+d_2t_3)\\
            &Q_4(t_4) = Q_{n_{0,4},d_2}(t_4):=\frac{1}{d_2}P(n_{0,4}+d_2t_4).
        \end{align}
        We note that $Q_i\in\Z[t_i]$ for all $i$, with $(Q_1(t_1),d_1)=(Q_2(t_2),d_2)=1$ by the squarefreeness condition (and similarly for $Q_3,Q_4$).

        We further split the sum over $t_i$ into two cases, depending on whether or not $\max_i\{Q_i(t_i)\}>T,$ for some parameter $T$ to be optimized later. We note that the total contribution for $t_i$ such $Q_i(t_i)$ is small ($\leq T$) is
        \begin{align}
            \ll_{P,\varepsilon} D^2\cdot N^\varepsilon\cdot (DT)^{\frac{4}{\deg P}},
        \end{align}
        which follows from the fact that
        \[
        \#\{t_i\leq N/d_{j}: Q_i(t_i)\leq T\}\ll \#\{x\leq N:P(x)\leq DT\}\ll (DT)^{\frac{1}{\deg P}},
        \]
        where we are already working under the assumption that $D,T$ will be much smaller compared to $N$. For the remaining $(t_1,t_2,t_3,t_4)$ (i.e., those for which at least one of $Q_i(t_i)>T$), we do some more case work, according to whether or not the GCD of $\max_i\{Q_i(t_i)\}$ and an appropriate $Q_j(t_j)$ is small. To this end, we may assume, w.l.o.g., that $\max_i\{Q_i(t_i)\}=Q_4(t_4)$ (the other cases being handled by symmetry). We let $\lambda$ be a parameter to be optimized and recall that we are currently considering the following:
        \begin{align}
        \label{Important_Equation}
            \sum_{\substack{{n_1,n_2,n_3,n_4\leq N}\\{d_1,d_2\leq D}}}&\one_{P(n_1)P(n_2)P(n_3)P(n_4)=\square}^{\text{SF}}\\
            &\ll_{P,\varepsilon}
             \sum_{d_1,d_2\leq D}
            \sum_{\substack{{n_{0,1},n_{0,2}\Mod{d_1}}\\{n_{0,3},n_{0,4}\Mod{d_2}}}}
            \sum_{\substack{{t_1,t_2\leq N/d_1}\\{t_3,t_4\leq N/d_2}\\{\max_i\{Q_i(t_i)\}=Q_4(t_4)>T}}}\one_{Q_1(t_1)Q_2(t_2)=Q_3(t_3)Q_4(t_4)}^{\text{SF}}
            +
            D^2N^{\varepsilon}(DT)^{\frac{4}{\deg P}}           
        \end{align}

            \subsubsection{The case when $(Q_4(t_4),Q_1(t_1)) \text{ and }(Q_4(t_4),Q_2(t_2))<Q_4(t_4)/\lambda$}
            
            We fix $t_4$ and consider $t_1,t_2$ in (\ref{Important_Equation}) such that both
            $$(Q_4(t_4),Q_1(t_1))<Q_4(t_4)/\lambda\;\;\;\&\;\;\;(Q_4(t_4),Q_2(t_2))<Q_4(t_4)/\lambda,$$
            for some $\lambda>0$ to be optimized later. It is necessary that $Q_1(t_1)Q_2(t_2)\equiv 0\Mod{Q_4(t_4)}$ and note that the number of such $t_1,t_2$ is bounded above by the number of $t_1,t_2$ such that $Q_i(t_i)\equiv 0\Mod{a_i}$ ($i=1,2$), where we run over all divisors $a_1,a_2$ of $Q_4(t_4)$ with $a_1a_2=Q_4(t_4)$. But since $(Q_4(t_4),Q_1(t_1))<Q_4(t_4)/\lambda\;\&\;(Q_4(t_4),Q_2(t_2))<Q_4(t_4)/\lambda$, we have that $a_i$ is necessarily greater than $\lambda$ (as $a_i<(Q_4(t_4),Q_i(t_i))<a_1a_2/\lambda$). Furthermore, for every such $a_i$ which is relatively prime to the leading coefficient of $Q_i$, Lemma \ref{Lagrange_Prime} yields the upper bound
            \[
            \ll_{P,\varepsilon} N^\varepsilon \left(\frac{N}{a_1d_1}+1\right)
            \left(\frac{N}{a_2d_1}+1\right),
            \]
            for the total number of such $t_1,t_2$. For the $a_i$ which are not relatively prime to the leading coefficient of $Q_i$, we need to consider primes $p\mid a_i$ for which $Q_i$ is the zero polynomial in $\F_p[x]$; for such primes, there are $p$ solutions mod $p$ and Lemma \ref{Lagrange_Prime} needs to be modified appropriately. We note that the leading coefficient of $Q_1(t_1)=\frac{1}{d_1}P(n_{0,1}+d_1t_1)$, say, is $d_1^k\alpha_P$ for some $k\geq 0$ and where $\alpha_P$ is the leading coefficient of $P$, so that we only need to consider primes dividing $d_1\alpha_P$. But the constant term of $Q_1(t_1)$ is $\frac{1}{d_1}P(n_{0,1})$ with $P(n_{0,1})\equiv0\Mod{d_1}$ and with $P(n_{0,1})$ squarefree. As such, $Q_1(t_1)$ will not be the zero polynomial for any $p\mid d_1$ and will be the zero polynomial for at most the primes dividing $\alpha_P,$ the number of which is uniformly bounded over all combinations of $d_i,n_{0,j},a_k$. Hence, the total contribution in the case where $Q_4(t_4)$ and $Q_1(t_1),Q_2(t_2)$ share a small divisor is bounded above by
            \begin{align}
                &\ll_{P,\varepsilon}
                N^{\varepsilon}
                \sum_{d_1,d_2\leq D}
                \sum_{n_{0,i}}\sum_{\substack{{t_4\leq N/d_2}\\{Q_4(t_4)>T}}}
                \sum_{\substack{{a_1,a_2>\lambda}\\{Q_4(t_4)=a_1a_2}}}
                \left(\frac{N}{a_1d_1}+1\right)
                \left(\frac{N}{a_2d_1}+1\right)\\
                &\ll_{P,\varepsilon}
                N^{\varepsilon}
                \sum_{d_1,d_2\leq D}
                \sum_{n_{0,i}}\sum_{\substack{{t_4\leq N/d_2}\\{Q_4(t_4)>T}}}
                \left(\frac{N^2}{Q_4(t_4)d_1^2}
                +\frac{N}{\lambda d_1}
                +1\right)\\
                &\ll_{P,\varepsilon}
                N^{\varepsilon}
                \sum_{d_1,d_2\leq D}
                \left(\frac{N^2}{T^{1-\frac{1}{\deg P}}d_1^2}
                +\frac{N^2}{\lambda d_1d_2}
                +\frac{N}{d_2}\right)\\
                &\ll_{P,\varepsilon}
                N^{\varepsilon}
                \left(\frac{N^2D}{T^{1-\frac{1}{\deg P}}}
                +\frac{N^2}{\lambda}
                +ND\right),
            \end{align}
            where we are working under the assumption that all parameters will be small powers of $N$.

        \subsubsection{The case when $(Q_4(t_4),Q_1(t_1))\geq Q_4(t_4)/\lambda$ or $(Q_4(t_4),Q_2(t_2))\geq Q_4(t_4)/\lambda$} 
        
        Similarly as above, we fix $t_4$ and wish to count the number of $t_1,t_2$ for which either $(Q_4(t_4),Q_1(t_1))>Q_4(t_4)/\lambda$ or $(Q_4(t_4),Q_2(t_2))>Q_4(t_4)/\lambda$. We simplify notation by writing $s_i:=Q_i(t_i)$ and, w.l.o.g.,\;we may assume that $g_{\{t_1,t_4\}}=g:=(s_1,s_4)>s_4/\lambda$. We write $s_1=gs_1^\prime$ and $s_4=gs_4^\prime$. Then both $s_1^\prime$ and $s_4^\prime$ are $\leq \lambda$ (recall that we are in the case where $s_4=Q_4(t_4)$ is maximal, so that $gs_1^\prime\leq gs_4^\prime$). Furthermore, we have
    \[
    s_4^\prime s_1-s_1^\prime s_4 = 0,
    \]
    which we rewrite as 
    \[
    s_4^\prime Q_1(t_1)-s_1^\prime Q_4(t_4)=0,
    \]
where we will think of $s_1^\prime$ and $s_4^\prime$ as constants up to $\lambda$. From Lemma \ref{New_BP}, we know that there are at most $\ll_\varepsilon N^{\frac{1}{3} + \varepsilon}$ such $t_1,t_4$ uniformly for $s_1^\prime,s_4^\prime$, unless $s_1^\prime=s_4^\prime$. But the latter corresponds to $t_1=t_4$ and we have already removed these trivial solutions from consideration. Hence, the total contribution to (\ref{Important_Equation}) in this case is bounded above by
    \begin{align}
    \label{large_gcd}
        &\ll_\varepsilon 
        N^\varepsilon
        \sum_{d_1,d_2\leq D}
        \sum_{n_{0,i}}
        \left(\lambda^2 N^{\frac{1}{3}}
        \cdot \frac{N}{d_1}\right)\\
        &\ll_\varepsilon
        \lambda^2 N^{1+\frac{1}{3} + \varepsilon}D,
    \end{align}
where the first line follows from the fact that we have counted the number of $(t_1,t_4)$ via Bombieri-Pila and where summing over $t_2$ (which loses a factor of $N/d_1$), fixes $t_3$.

To conclude this subsection, we have that the total contribution from small $d_1,d_2$ is bounded above by
\begin{align}
    \label{small_GCD_total}
    \ll_{P,\varepsilon}
    N^{\varepsilon}
        \left(\frac{N^2D}{T^{1-\frac{1}{\deg P}}}
                +\frac{N^2}{\lambda}
                +\lambda^2N^{1+\frac{1}{3}}D    
                +
                D^2(DT)^{\frac{4}{\deg P}}
                \right).
\end{align}

\subsection{Large $d_1,d_2$}

It remains to consider large GCD. Recall the equation
\begin{equation}
\label{mainEquation}
    d_2^2 P(n_1)P(n_2) = d_1^2 P(n_3)P(n_4) 
\end{equation}
with $n_1,n_2,n_3,n_4\leq N$ and where either $d_1:=(P(n_1),P(n_2))>D$ or $d_2:=(P(n_3),P(n_4))>D$. Without loss of generality, we may assume that $d_1>D$. Our first task is to count the number of $n_1,n_2\leq N$ such that $d_1=(P(n_1),P(n_2))>D$. We will show that there are $\ll_{P,\varepsilon}N^{2+\varepsilon}/D$ such $n_1,n_2\leq N$ and then show that the number of $n_3,n_4\leq N$ remaining loses an additional $N^{\frac{1}{3}}.$ The former is accomplished by trivial divisors bounds and a bootstrapping argument; the latter is accomplished by Bombieri-Pila, in much the same way as we did before.

\subsubsection{Counting $n_1,n_2$ with $D<d_1\ll_P N$}
The number of $n_1,n_2\leq N$ such that $(P(n_1),P(n_2))=:d_1$ is larger than $D$ but $\ll_P N$ (for some implied constant to be specified later) is bounded trivially by
    \begin{align*}
        \sum_{D<d_1\ll_P N}&\Biggr(\#\{x\leq N: P(x)\equiv 0\Mod{d_1}\}\Biggr)^2\\
        &\ll_{P,\varepsilon}
        N^{\varepsilon}\sum_{D<d_1\ll_P N}\left(\frac{N}{d_1}+1\right)^2\\
        &\ll_{P,\varepsilon} \frac{N^{2+\varepsilon}}{D},
    \end{align*}
which follows from Lemma \ref{Lagrange_Prime}.

For $n_1,n_2$ with $d_1=(P(n_1),P(n_2))\gg N$, we require a bootstrapping argument, whose details we spell out below. This argument appears in a related but different context in \cite[Theorem 1.8]{Granville-twists}.

\subsubsection{Counting $n_1,n_2$ with $d_1\gg_P N$}

Suppose $d_1|P(n_1),P(n_2),$ for some $n_1,n_2\leq N$, with $d_1\gg_P N$. Recall that $P$ is a polynomial with integer coefficients, which is either a product of an arbitrary number of distinct linear factors over $\Z$ or (irreducible) of degree $2$. We claim that $d_1$ is composite. In the case where $P$ is a product of linear factors, this is clear as each prime factor of $P(n_i)$ will be $\ll_P N$. In the case where $P$ is a quadratic polynomial, say $P(x)=ax^2+bx+c$, we have that $d_1 | P(n_1)-P(n_2) = a(n_1^2 - n_2^2) + b(n_1-n_2) = (n_1-n_2)(a(n_1+n_2)+b),$ so that every prime factor of $d_1$ divides either $n_1-n_2\leq N$ or $a(n_1+n_2)+b\ll_P N$. Write $d_1=p_1p_2\cdots p_k$, where $p_1<p_2<\dots<p_k$. If there exists $p_i>D$, set $d_0:=p_i\ll_P N$; if no such $p_i$ exists, then there exists some minimal $i\leq k$ such that $d_0:=p_1p_2\cdots p_i>D$ with $d_0\leq D^2$. We have just shown that if $P((n_1),P(n_2))=d_1\gg_P N$, then there exists $D<d_0\ll_P \max\{D^2,N\}$ such that $d_0|d_1;$ as such, the number of choices for $n_1\leq N$ and $n_2\leq N$ with $d_1\gg_P N$ is certainly
\begin{equation}
    \ll_{P,\varepsilon} 
    N^{\varepsilon}
    \sum_{D< d_0\ll_P \max\{D^2, N\}}
    \left(\frac{N}{d_0} + 1\right)^2 
    \ll_{\varepsilon} \frac{N^{2+ \varepsilon}}{D},
    \end{equation}
where we are still working under the assumption that $D$ will be a very small power of $N$.

\subsubsection{Counting the remaining $n_3,n_4$}

In the previous two subsections, we have shown that the total number of $n_1,n_2\leq N$ with $(P(n_1),P(n_2))=d_1\gg_P N$ is $\ll_{P,\varepsilon}N^{2+\varepsilon}/D$. Given one of these $N^{2+\varepsilon}/D$ choices of $n_1$ and $n_2$ (which determine $d_1$), it now remains to bound
    \begin{align*}
        \sum_{n_3,n_4\leq N}\one_{P(n_3)P(n_4)=d_2^2\frac{P(n_1)P(n_2)}{d_1^2}}^{\text{SF}}.
    \end{align*}
Suppose $n_3,n_4$ are such that $\frac{1}{d_2^2}P(n_3)P(n_4)=\frac{1}{d_1^2}{P(n_1)P(n_2)}$. Then, $P(n_3)=d_2a$ and $P(n_4)=d_2b$, for some $a,b$ with $ab=\frac{1}{d_1^2}{P(n_1)P(n_2)}$. Thus, $n_3$ and $n_4$ are such that $bP(n_3)-aP(n_4)=0$; and so,
\begin{align*}
       \sum_{n_3,n_4\leq N}&\one_{\frac{P(n_3)P(n_4)}{d_2^2}=\frac{P(n_1)P(n_2)}{d_1^2}}^{\text{SF}}\\
       &\leq
       \sum_{\substack{{a,b}\\{ab=\frac{P(n_1)P(n_2)}{d_1^2}}}}
       \#\{n_3,n_4\leq N: aP(n_3)-bP(n_4)=0\}\\
       &\ll_{P,\varepsilon} N^{\frac{1}{3}+\varepsilon},
\end{align*}
where the last line follows from Lemma \ref{New_BP}, together with the fact that there are a divisor bounded number of choices for $a,b$ (with $a\neq b,$ as $n_1\neq n_2$). Hence, the total contribution for large $d_1,d_2$ is
\[
\ll_{P,\varepsilon} \frac{N^{2+\frac{1}{3}+\varepsilon}}{D}.
\]

\subsection{Picking Parameters}

It remains to choose appropriate parameters $T,D,\lambda$. The total number of non-trivial solutions is bounded above by
\[
\ll_{P,\varepsilon}
    N^{\varepsilon}
        \left(\frac{N^2D}{T^{1-\frac{1}{\deg P}}}
                +\frac{N^2}{\lambda}
                +\lambda^2N^{1+\frac{1}{3}}D    
                +
                D^2(DT)^{\frac{4}{\deg P}}
                +
                \frac{N^{2+\frac{1}{3}}}{D}
                \right).
\]
Let $d = \deg P$ and take $\delta_P =\delta = \frac{1}{2025}$, with $T=N^{\frac{d(2\delta+\frac{1}{3})}{d-1}}, D=N^{\delta+\frac{1}{3}},$ and $\lambda = N^{\delta}$; this choice of parameters yields an error of size
\[
\ll_{P,\varepsilon}
N^{2-\delta+\varepsilon},
\]
for all $d\geq 3$ and all $\varepsilon>0$. We remark that the exponent $\frac{1}{3}$ in the third and fifth error terms, which come from Lemma \ref{New_BP}, can likely be reduced to $\frac{1}{\deg P}$ with more work\footnote{One would need to show that the curves considered in the proof of Proposition \ref{counting_squares} are absolutely irreducible.}, but this improvement is not necessary for our desired application. We have also not chosen the parameters optimally: one can do better for $d\geq 4$, with an error of size $\ll_{P,\varepsilon}
N^{\frac{23}{12}+\varepsilon}$. For $d=2$, the $\frac{1}{3}$ exponents from the third and fifth error terms are not there, as there are at most $\ll_{P,\varepsilon}N^\varepsilon$ integral points of low height on genus $0$ curves (recall Lemma \ref{Quadratic_Curves}), and we can take $\delta_P = \frac{1}{4}$, $D=\lambda=N^{\frac{1}{4}}$, and $T=N$, for an error of size $\ll_{P,\varepsilon} N^{\frac{7}{4}+\varepsilon}$.
\end{proof}

We are now ready to prove Theorem \ref{main_thm}.

\section{Proof of Theorem \ref{main_thm}}
\label{Section_Proof}

With Proposition \ref{counting_squares} in tow, it is now simple to show that $(S_{p,N})_p$ satisfies the hypotheses to Lemma \ref{McLeish_CLT}. For convenience to the reader, we provide the details below. First, recall that we are considering the following random variables:
\[
M_N:=\sum_{n\leq N}f(P(n))\;\;\&\;\;M_{p,N}:=\sum_{\substack{{n\leq N}\\{P^+(P(n))=p}}}f(P(n))
\]
and their normalized counterparts:
\[
S_N:=\frac{1}{\sqrt{\E\left[M_N^2\right]}}M_N\;\;\&\;\;S_{p,N}:=\frac{1}{\sqrt{\E\left[M_N^2\right]}}M_{p,N}.
\]
We wish to verify conditions (1), (2), and (3) of Lemma \ref{McLeish_CLT}, with $k_N=\max\{p : p \mid P(j), j \leq N\}$.

Note that condition (1) of Lemma \ref{McLeish_CLT} is trivial; indeed:
\begin{align*}
    \sum_{p\leq k_N}\E\left[S_{p,N}^2\right]
    &=
    \sum_{p\leq k_N}\frac{1}{\E\left[M_N^2\right]}\E\left[M_{p,N}^2\right]\\
    &=
    \frac{1}{\E\left[M_N^2\right]}\sum_{p\leq k_N}\sum_{\substack{{n_1,n_2\leq N}\\{P^+(P(n_1))=p=P^+(P(n_2))}}}\E\left[f(P(n_1))f(P(n_2))\right]\\
    &=
    \frac{1}{\E\left[M_N^2\right]}\left(\sum_{p\leq k_N}\#\{n\leq N:\text{$P(n)$ is squarefree with $P^+(P(n))=p$}\} + O_P(1)\right),
\end{align*}
which follows from the fact that Rademacher RMFs are supported on squarefree $n$, together with the fact that $\E\left[f(m)f(n)\right]=1$ whenever $m=n$ and $0$ otherwise, for $m,n$ squarefree. The $O_P(1)$ term comes from the fact that $P(x)$ may only be injective for $x$ sufficiently large. But this last line converges (is precisely equal, for $N$ sufficiently large) to $1$, as
\begin{align*}
        \sum_{p\leq k_N}\E\left[S_{p,N}^2\right]
        &=
        \frac{1}{\E\left[M_N^2\right]}\left(\sum_{p\leq k_N}\#\{n\leq N:\text{$P(n)$ is squarefree with $P^+(P(n))=p$}\}+O_P(1)\right)\\
        &=
        \frac{1}{\E\left[M_N^2\right]}\Big(\#\{n\leq N:\text{$P(n)$ is squarefree}\}+O_P(1)\Big),
\end{align*}
with
\begin{align*}
    \E\left[M_N^2\right]
    &=
    \sum_{n_1,n_2\leq N}\E\left[f(P(n_1))f(P(n_2))\right]\\
    &=
    \#\{n\leq N: \text{$P(n)$ is squarefree}\}+O_P(1),
\end{align*}
for the same reasons as above. We note that condition (1) is satisfied for all $P\in\Z[x]$.

Similarly as above, condition (3) of Lemma \ref{McLeish_CLT} is easily verifiable:
\begin{align*}
    \sum_{\substack{{p,q\leq k_N}\\{p\neq q}}}\E\left[S_{p,N}^2S_{q,N}^2\right]
    &=
    \frac{1}{\left(\E\left[M_N^2\right]\right)^2}\sum_{\substack{{p,q\leq k_N}\\{p\neq q}}}\E\left[M_{p,N}^2M_{q,N}^2\right]\\
    &=
    \frac{1}{\left(\E\left[M_N^2\right]\right)^2}\sum_{\substack{{p,q\leq k_N}\\{p\neq q}}}\E\left[\Big(\sum_{\substack{{m,n\leq N}\\{P^+(P(m))=p,P^+(P(n))=q}}}f(P(m))f(P(n))\Big)^2\right]\\ 
    &=
    \frac{1}{\left(\E\left[M_N^2\right]\right)^2}\sum_{\substack{{p,q\leq k_N}\\{p\neq q}}}\sum_{\substack{{m_1,m_2,n_1,n_2\leq N}\\{P^+(P(m_i))=p,P^+(P(n_i))=q}}}\E\left[f(P(m_1))f(P(m_2))f(P(n_1))f(P(n_2))\right]\\ 
    &=
    \frac{1}{\left(\E\left[M_N^2\right]\right)^2}\sum_{\substack{{p,q\leq k_N}\\{p\neq q}}}\left(\sum_{\substack{{m_1,m_2,n_1,n_2\leq N}\\{P^+(P(m_i))=p,P^+(P(n_i))=q}}}\one_{P(m_1)P(m_2)P(n_1)P(n_2)=\square}^{\text{SF}}\right)\\
    &=
    1+o(1)+\frac{1}{\left(\E\left[M_N^2\right]\right)^2}\left(
    \sum_{\substack{{p,q\leq k_N}\\{p\neq q}}}\sum_{\substack{{m_1,m_2,n_1,n_2\leq N}\\{P^+(P(m_i))=p,P^+(P(n_i))=q}\\{m_1\neq m_2,n_1\neq n_2}}}\one_{P(m_1)P(m_2)P(n_1)P(n_2)=\square}^{\text{SF}}\right),
\end{align*}
where the $o(1)$ term takes into account the non-injective values of $P$. Since $\E\left[M_N^2\right]\sim \kappa_PN$, it suffices to show that 
\begin{align*}
    \sum_{\substack{{p,q\leq k_N}\\{p\neq q}}}&\sum_{\substack{{m_1,m_2,n_1,n_2\leq N}\\{P^+(P(m_i))=p,P^+(P(n_i))=q}\\{m_1\neq m_2,n_1\neq n_2}}}\one_{P(m_1)P(m_2)P(n_1)P(n_2)=\square}^{\text{SF}}\\
    &\leq
   \sum_{\substack{{m_1,m_2,n_1,n_2\leq N}\\{P^+(P(m_1))=P^+(P(m_2)), P^+(P(n_1))=P^+(P(n_2))}\\{m_1\neq m_2,n_1\neq n_2}}}\one_{P(m_1)P(m_2)P(n_1)P(n_2)=\square}^{\text{SF}}\\
    &=
    o(N^2),
\end{align*}
as $N\rightarrow\infty$. But note that this is bounded above by the number of non-trivial solutions from Proposition \ref{counting_squares}: we already have that $m_1\neq m_2$ and $n_1\neq n_2$, but if $m_1=n_1$ say, then the condition that $P(m_1)P(m_2)P(n_1)P(n_2)=\square$ with each $P(m_i),P(n_i)$ square forces $n_1=n_2$ (apart from the finitely-many cases where $P$ does not have a unique image); that is,
\begin{align*}
   \sum_{\substack{{m_1,m_2,n_1,n_2\leq N}\\{P^+(P(m_1))=P^+(P(m_2)), P^+(P(n_1))=P^+(P(n_2))}\\{m_1\neq m_2,n_1\neq n_2}}}&\one_{P(m_1)P(m_2)P(n_1)P(n_2)=\square}^{\text{SF}}\\
    &\leq
    \sum_{\substack{{m_1,m_2,n_1,n_2\leq N}\\{m_i\neq n_j}}}\one_{P(m_1)P(m_2)P(n_1)P(n_2)=\square}^{\text{SF}}\\
    &=
    o(N^2),
\end{align*}
as desired.

It remains to show that $(S_{p,N})_p$ satisfies the fourth moment condition (See Remark \ref{fourth_moment_condition}). To this end, consider the following:
\begin{align*}
    \sum_{p\leq k_N}\E\left[S_{p,N}^4\right]
    &=
    \frac{1}{\left(\E\left[M_N^2\right]\right)^2}\sum_{p\leq k_N}\sum_{\substack{{n_1,n_2,n_3,n_4\leq N}\\{P^+(P(n_i))=p}}}\E\left[f(P(n_1))f(P(n_2))f(P(n_3))f(P(n_4))\right]\\
    &=
    \frac{1}{\left(\E\left[M_N^2\right]\right)^2}\left(\sum_{p\leq k_N}\sum_{\substack{{n_1,n_2,n_3,n_4\leq N}\\{P^+(P(n_i))=p}}}\one_{P(n_1)P(n_2)P(n_3)P(n_4)=\square}^{\text{SF}}+O_P(1)\right)\\
    &\lesssim
    \frac{1}{\kappa_P^2N^2}\sum_{p\leq k_N}\sum_{\substack{{n_1,n_2,n_3,n_4\leq N}\\{P^+(P(n_i))=p}}}\one_{P(n_1)P(n_2)P(n_3)P(n_4)=\square}^{\text{SF}}.  
\end{align*}
We need to show that
\[
\sum_{p\leq k_N}\sum_{\substack{{n_1,n_2,n_3,n_4\leq N}\\{P^+(P(n_i))=p}}}\one_{P(n_1)P(n_2)P(n_3)P(n_4)=\square}^{\text{SF}}
=
o(N^2),
\]
but, by Proposition \ref{counting_squares}, we only need to consider the diagonal solutions (i.e., those $n_i$ which are equal in pairs); in particular,
\begin{align*}
    \sum_{p\leq k_N}&\sum_{\substack{{n_1,n_2,n_3,n_4\leq N}\\{P^+(P(n_i))=p}}}\one_{P(n_1)P(n_2)P(n_3)P(n_4)=\square}^{\text{SF}}\\
    &\ll
    \sum_{p\leq k_N}\Biggr(\sum_{\substack{{n\leq N}\\{P^+(P(n))=p}}}1\Biggr)^2 + o(N^2)\\
    &\ll
    \sum_{p\leq \log \log N}\Biggr(\sum_{\substack{{n\leq N^{\deg P}}\\{n:\text{ $p$-smooth}}}}1\Biggr)^2
    +
    \sum_{\log \log N<p\leq k_N}\Biggr(\sum_{\substack{{n\leq N}\\{P(n)\equiv 0\Mod{p}}}}1\Biggr)^2
    +
    o(N^2).
\end{align*}
For the smooth sum, we use the fact that there are very few $p$-smooth numbers of size $N^{\deg P}$ with $p\leq\log\log N$. Indeed, if $\psi(x,y)$ denotes the number of $y$-smooth integers up to $x$, then we know that $\psi(x,y)=xu^{-(1+o(1))u}$ uniformly in $u\leq y^{1-\epsilon}$, for any $\epsilon>0,$ where $u=\log x / \log y;$ see \cite[Corollary 1.3]{HildTen}, for example. This results in an error of size $\ll_{\varepsilon} N^\varepsilon$. For the second sum, we split into cases, depending on the size of $p$. For the primes $p\ll_P N$ we use the second divisor bound from Lemma \ref{Lagrange_Prime} which yields the upper bound $\ll \left(N^2/p^2 + N/p +1\right)^2$ on the inner sum. Summing over the small primes then yields an admissible error. For the larger primes, note that there are $\ll_P N$ primes $p\gg N$ such that $p\mid P(n)$ for some $n\leq N$ (since each such value of $P(n)$ is divisible by $\ll\deg P$ of such large primes). For each of these primes, the number of $n\leq N$ such that $P(n)\equiv 0\pmod p$ is $\ll 1$, again by Lemma \ref{Lagrange_Prime}, which yields a total error of $\ll_P N$. Thus,
\[
\sum_{p\leq k_N} \E\left[S_{p,N}^4\right] \rightarrow 0,
\]
as $N\rightarrow\infty,$ as desired, which completes the proof of Theorem \ref{main_thm}.

\section{Large Fluctuations}
\label{large_values_section}

In this final section, we prove Theorem \ref{large_values_theorem} on large fluctuations of $\sum_{n\leq N}f(n^2+1)$. We follow the strategy as in the Steinhaus case \cite[Theorem 1.3]{Oleksiy-RandomChowla}. We first 
give a detailed outline of the proof. 

\subsection{Proof Strategy} 

The key idea is to utilize the existence of many positive integers $n$ such that $n^2 + 1$ has a very large prime factor (see Lemma \ref{posProp}): if $n\leq N$ is such that $n^2+1$ has a prime factor $p$ bigger than $N\log N$, say, then this prime is the largest prime factor of $n^2+1$ and $(n^2 + 1)/p$ has all of its prime factors less than $N/\log N$; and so, we can ``pull out''\footnote{Recall that $f$ is supported on square-free integers.} $f(p)$ from $f(n^2 + 1)$ and write the initial sum as 
\begin{equation}
    \sum_{n\leq N} f(n^2 + 1) = \sum_{p\gg N\log N} f(p)\sum_{n\leq N : p\mid n^2 + 1} f\left(\frac{n^2 + 1}{p}\right) + \text{the remaining terms}.\label{eq_strategy}
\end{equation}
Conditioning on the primes $\ll N/\log N$, the above sum which, we denote by $S(N)$, may be viewed as $S(N) \approx \sum_{p\gg N\log N} f(p)c_p(N)$, where $c_P(N):=\sum_{n\leq N : p\mid n^2 + 1} f\left(\frac{n^2 + 1}{p}\right)$ are ``fixed'' coefficients. This is a weighted sum of independent and identically distributed random variables, which should have approximately Gaussian behavior, as long as the weights are not too irregular. 

The second key idea is to sample $S(N)$ at multiple, fairly spread out scales $x_1, x_2, \ldots, x_k$ with corresponding sums $S_i := S(x_i)$ for $1\leq i\leq k$. In order to work with these sums simultaneously, it becomes crucial to carry out the conditioning more carefully, for we want to avoid viewing a given $f(p)$ as random at some scale but conditioned on at a different scale. We shall sweep this detail under the rug for this outline. 

We will show that, with probability $\approx 1$ over the choices for $f(q)$ that have been conditioned on, the remaining terms in \eqref{eq_strategy} can be ignored at most scales ($\lfloor 0.99k\rfloor$ of them, say). Upon conditioning, we then have $\lfloor 0.99k\rfloor$ approximately Gaussian  sums which ``should'' behave independently from one-another, given that the scales $x_i$ are spread out. In fact, upon carrying the conditioning carefully, the sums $S_i$ will be truly independent. We package all of these sums in one multivariate distribution and compare it to the corresponding multivariate normal distribution via the following lemma.

\begin{lem}[Normal Approximation Result \cite{Harper_Fluctuations}]
\label{normApproxLemma}
    Let $m \in \N $ and let  $S\subset \N$ be finite and nonempty. Suppose that for each $1 \leq \ell\leq m$ and $h \in S$ we are given a deterministic coefficient $c(\ell,i) \in \C$. Finally, suppose that $(V_\ell)_{1\leq \ell\leq m}$ is a sequence of independent, mean zero, complex-valued random variables, and let $Y = (Y_i)_{i\in S}$ be the $\#S$-dimensional random vector with components $Y_i := \Re(\sum_{\ell=1}^m c(\ell,i)V_\ell)$. If $Z = (Z_i)_{i\in S}$ is a multivariate normal random vector with the same mean vector and covariance matrix as $Y$, then for any $u\in \R$ and any small $\eta > 0$, we have
\begin{align*}
    \mathbb{P}\left(\max_{i\in S} Y_i\leq u\right) 
    &\leq
    \mathbb{P}\left(\max_{i\in S} Z_i\leq  u + \eta\right)\\
 &\quad+
O\left(\frac{1}{\eta^2}\sum_{i,j\in S}\sqrt{
\sum_{\ell=1}^m
|c(\ell, i)|^2|c(\ell,j)|^2\E\left [|V_\ell|^4\right]}
+
\frac{1}{\eta^3}\sum_{\ell=1}^m \E\left [|V_\ell|^3\right]\left(\sum_{i\in S}|c(\ell,i)|\right)^3
\right).
\end{align*}
\end{lem}

Using the notation of the lemma, we will choose: $V_\ell$ to be, essentially, the $f(p)$ for $p\gg N\log N$; $S$ to be the set of the indices of the $\lfloor 0.99k\rfloor$ scales; $c(\ell,i)$ to be the coefficient $c_p(x_i)$ from the sentence following \eqref{eq_strategy} (where $\ell$ corresponds to $p$). Note that both the set $S$ and the coefficients $c(\ell,i)$ are indeed deterministic upon conditioning. 

Now that we have passed to the normal approximations for the $\lfloor 0.99k\rfloor$ sums, the maximum of these independent mean zero and unit variance Gaussians gets as large as roughly $\sqrt{2\log k}$ with probability $\approx 1$ over the random $f(p)$ that have not been conditioned on; this is the content of the next lemma. Choosing $k$ to be roughly $\log N$ and carrying out this process for infinitely many disjoint sets of scales gives the desired fluctuations of size $\sqrt{\log\log N}$ with probability $\approx 1$.

\begin{lem}[Normal Comparison Result \cite{Harper_Fluctuations}] \label{normComparisonLemma} Suppose that $n \geq 2$, and that $\varepsilon \geq 0$ is sufficiently small (i.e., less than a certain small absolute constant). Let $X_1,\dots, X_k$ be
mean zero, variance one, jointly normal random variables\footnote{We say that two random variables $X,Y$ are \textit{jointly normal} if $aX+bY$ is a normal random variable for all $a,b\in\R$.}, and suppose $\E[X_iX_j ]\leq \varepsilon$ whenever $i\neq j$. Then, for any $100\varepsilon \leq \delta \leq 1/100$ (say), we have
\[
\mathbb{P}\left(\max_{1\leq i\leq k} X_i\leq \sqrt{(2-\delta)\log k}\right) \leq  \e^{-\Theta(k^{\delta/20}/\sqrt{\log k})} + k^{-\delta^2/50\varepsilon}.
\]
\end{lem}

\subsection{Localization}

As remarked by Harper \cite[Theorem 2]{Harper_Fluctuations}, it suffices to prove the existence of large fluctuations locally in order to obtain large fluctuations globally; as such, Theorem \ref{large_values_theorem} will follow from the following localized version.

\begin{lem}\label{large_values_lemma}
For $X$ sufficiently large, we have 
\begin{equation}
    \max_{N\in [X, X^{(\log X)^2}]} \frac{1}{\sqrt N}\left\lvert\sum_{n\leq N} f(n^2 + 1)\right\rvert \gg \sqrt{\log\log X}\label{local_large_val}
\end{equation}
with probability $1 - O((\log\log X)^{-1/50})$.
\end{lem}

The fact that Lemma \ref{large_values_lemma} implies Theorem \ref{large_values_theorem} is trivial, but we provide a proof for convenience to the reader. 

\begin{proof} [Proof that Lemma \ref{large_values_lemma} implies Theorem \ref{large_values_theorem}]
    To begin, note that \eqref{local_large_val} fails with probability $\ll (\log\log X)^{-1/50}$. Summing these probabilities over a suitably sparse sequence of values of $X$ yields a convergent series. By the Borel-Cantelli lemma, the probability that \eqref{local_large_val} fails for infinitely many of the chosen values of $X$ is $0$. Hence, there almost surely exist arbitrarily large values of $N$ for which $\frac{1}{\sqrt{N}}\Big|\sum_{n\leq N} f(n^2 + 1)\Big|\gg \sqrt{\log\log N}$, noting that $\log\log N\asymp \log \log X$ for $N\in[X, X^{(\log X)^2}$]. 
\end{proof}

The rest of this section is devoted to proving Lemma \ref{large_values_lemma}.

\subsection{Scales and conditioning}

Let $X$ be large and let $x_1, x_2, \ldots, x_k\in [X, X^{(\log X)^2}]$ be defined by $x_i = X^{i(\log 3i)^2}$ for $i=1, 2, \ldots, k$, where $k = \lfloor \log X\rfloor$. Define $\mathcal E_{x} = \{p\gg x\log x : p \mid n^2 + 1 \text{ for some } n\leq x\}$. Further define $\mathcal A_i = \mathcal E_{x_i} \setminus \{p\gg x_i\log x_i : p\mid n^2 + 1 \text{ for some } n\leq x_{i-1}\}$ (this is to ensure there is no interaction between the different scales). Define $\mathcal A = \mathcal A_X = \bigcup_{i=1}^k \mathcal A_i$. We now prove some basic properties about the sets $\mathcal A_i$. 

\begin{lem}[Properties of the sets $\mathcal A_i$] Let $X$  be large enough. We have that \begin{itemize}
    \item[(i)] the intersections $\mathcal A_i\cap \mathcal A_j$ are empty for all $i\neq j$,
    \item[(ii)] each $\mathcal A_i$ is of size $\asymp x_i$,
    \item[(iii)] no two distinct primes in $\mathcal A_i$ both divide $n^2 + 1$ for some $n\leq x_i$. 
\end{itemize} \label{AsetProperties}
\end{lem}

Properties (i) and (iii) are immediate. Property (ii) is a straightforward consequence of the following lemma. 

\begin{lem}\label{posProp}
    There exist a positive proportion of positive integers $n$ such that $n^2 + 1$ is square-free and $P^+(n^2 + 1)\gg n\log n$ for some suitably small implied absolute constant.
\end{lem}

\begin{proof}
    In \cite{maynard2021a} it is shown that the density of $n$ such that $P^+(n^2 + 1)>n$ is at least $1/2$ (see also the improvements in \cite{JTNB_2020__32_3_891_0}). The density of $n$ such that $n^2 + 1$ is squarefree is $\prod_{p\equiv 1\pmod 4} (1-2/p^2)\approx 0.894$. It follows that there is a positive proportion of $n$ such that $n^2 + 1$ is square-free and $P^+(n^2 + 1) > n$, because the densities add up to more than $1$ and thus the corresponding sets cannot be disjoint. Let the set of such $n\in[N/\log N, N]$ be called $\mathcal N$.
   
   Following the bootstrapping argument in \cite{maynard2021a} due to Granville, we show that there is a positive proportion of $n$ such that $n^2 + 1$ is square-free as well as $P^+(n^2 + 1)\gg n\log n$. Indeed, fix $\delta>0$ and consider the set 
   \begin{equation}\mathcal S = \left\{n\in \left[\frac{N}{\log N}, N\right] : P^+(n^2 + 1) < \delta n\log n \text{ and } n^2 + 1 \text{ square-free}\right\}.\end{equation} Assume that $\mathcal S$ has full density in the set of $n$ in the same interval for which $n^2 + 1$ is square-free. Intersecting $\mathcal S$ with the set $\mathcal N$ gives that $\lvert \mathcal N\cap\mathcal S\rvert\geq cN$ for some absolute constant $c>0$ and $N$ large enough. But now the set of largest prime factors of $n\in\mathcal N\cap\mathcal S$ is contained in $[N/\log N, \delta N\log N]$ by the assumption on $\mathcal S$. This set is also of size at least $cN/2$, since each such prime factor $p$ divides at most two values of $n^2 + 1$. This gives $c/2 < \delta$ which is not true for $\delta$ chosen suitably small. It follows that the set of $n\in [N/\log N, N]$ such that $P^+(n^2 + 1)\gg n\log n$ and $n^2 + 1$ is square-free, for some appropriately chosen implicit constant, is of positive proportion.
\end{proof}

\begin{rem}\label{rem_quadr}
    The above proof is the only ingredient that does not directly work for general quadratic polynomials (i.e., those whose density of square-free values is smaller than $1/2$). However, one would be able to run the proof of Maynard-Rudnick \cite{maynard2021a} from the very beginning to get what is needed in the general case. 
\end{rem}

Property (ii) is a localization argument applied to Lemma \ref{posProp} and follows from the fact that a given prime $p\in\mathcal A_i$ divides at most two values of $n^2 + 1$ for $n\leq x_i$ (so we lose at most a factor of $2$ in the ``positive proportion'' statements). More precisely, by Lemma \ref{posProp}, we have that there are $\gg x$ values of $n\in[cx, x]$ such that $n^2 + 1$ is square-free and $P^+(n)\gg (cx)\log(cx)\gg x\log x$, as long as $c>0$ is a small enough absolute constant. Considering the set of the largest prime factors of $n^2 + 1$ for such $n\in[cx_i, x_i]$ and noting that the set we remove from $\mathcal E_{x_i}$ in the definition of $\mathcal A_i$ is of size $\ll x_{i-1} = o(x_i)$, yields the lower bound in (ii). Noting that the upper bound $\mathcal A_i\ll x_i$ is trivial finishes the proof of property (ii).

We are now in a position to produce large fluctuations; we begin by decomposing our initial sum at scale $x_i$ into the following three sums: \begin{equation}
    \sum_{n\leq x_i} f(n^2 + 1) = \sum_{\substack{p\in\mathcal A_i}}\sum_{\substack{n\leq x_i \\ p\mid n^2 + 1 \\ q\mid (n^2 + 1)/p \implies q\not\in\mathcal A}} f(n^2 + 1) + \sum_{\substack{n\leq x_i \\ p\mid n^2 + 1 \text{ for some } p\in\mathcal A\setminus \mathcal A_i}} f(n^2 + 1) + \sum_{\substack{n\leq x_i \\ p\mid n^2 + 1\implies p\not\in\mathcal A}} f(n^2 + 1) \label{largeVal1}
\end{equation}

Denote the first sum in the above by $S_{i,1}$. We will condition on the values of $f(q)$ for all $q$ not in $\mathcal A$, making the first sum a linear combination of independent random variables. However, we first deal with the second and third sums in the next subsection.

\subsection{Ignoring the second and third sums}\label{2nd3rdsums} Let us first show the second sum can be ignored at all scales due to density reasons. The second moment of the sum (due to $\mathbb E_f f(n^2+1)f(m^2 + 1)\neq 0$ if and only if $n=m$, assuming that $n^2 + 1$ and $m^2 + 1$ are square-free) is equal to the number of $n\leq x_i$ such that $p\mid n^2 + 1$ for some $p\in \bigcup_{j<i} \mathcal A_j$. This is at most \begin{align}
    \ll & \sum_{j=1}^{i-1} \sum_{p\in\mathcal A_j}\left(\frac{x_i}{p} + 1\right) \\ \ll & x_i\sum_{j=1}^{i-1} \left(\sum_{p\in\mathcal A_j} \frac{1}{p} + \lvert\mathcal A_j\rvert\right)
    \\ \ll &  x_i\sum_{j=1}^{i-1} \left(\frac{x_j}{x_j\log x_j} + x_j \right) \\ \ll & \frac{x_i}{\log X}, \label{littleOh}
\end{align}
where the last inequality follows by definition of the sampling points $x_j$ and $\log x_j\gg \log X$. 

Markov's inequality and the union bound implies that the probability that the second sums are larger than $\sqrt x_i(\log \log x_i)^{1/100}$ for every $i=1, 2, \ldots, k$ is at most $O(k/(\log X(\log \log x)^{1/50}))$. This is small enough given that $k = \lfloor\log X\rfloor$. 

Now we deal with the third sums, which crucially do not depend on the values of $f(p)$ for $p\in\mathcal A$. A second moment estimate yields with probability $1-O((\log \log X)^{-1/50})$ a large subset $S$ (i.e. of cardinality $\geq 0.99k$) of indices in $\{1, 2, \ldots, k\}$ such that the third sums have typical behavior at all of the scales in this subset; that is, for every $i\in S$, we have that the probability that the third sum is greater than $\sqrt{x_i}(\log \log x_i)^{1/100}$ is $O(1/(\log \log X)^{1/50})$. By the union bound, the number of indices for which the third sum is larger than $\sqrt{x_i}(\log\log x_i)^{1/100}$ is at most $O(k/(\log\log X)^{1/50})$, from where the desired conclusion follows. Note that this set $S$ only depends on the values of $f(q)$ for $q\not\in\mathcal A$, and thus may be viewed as ``fixed'' upon conditioning.

\subsection{Creating large fluctuations} 

First note that \eqref{littleOh} along with Lemma \ref{AsetProperties} also shows that the first sum $S_{i,1}$ from \eqref{largeVal1} contains a good chunk of the entire sum. 

\begin{lem}\label{goodChunk}
    There exist $\gg x_i$ values of $n\leq x_i$ such that $n^2 + 1$ is square-free as well as $p\mid n^2 + 1$ for some $p\in\mathcal A_i$, but $(n^2 + 1)/p$ has no prime factors in $\mathcal A$. 
\end{lem}

\begin{proof}
We have already established that there are $\gg x_i$ values of $n\leq x_i$ such that $n^2 + 1$ is square-free and $p\mid n^2 + 1$ for some $p\in\mathcal A_i$. The number of those $n\leq x_i$ for which there is an additional prime divisor of $n^2 + 1$ in $\mathcal A$ is $o(x_i)$ (as $X\to\infty$) by \eqref{littleOh}. 
\end{proof}

Next we condition on the values of $f(q)$ for all of the primes $q$ not in $\mathcal A$. For any realizations of the latter for which a suitable set $S$ as in the conclusion of the previous subsection exists (which happens with probability $1-O(1/(\log \log X)^{1/50})$), we pass from the sums $S_{i,1}$ with $i\in S$ to their normal approximations.

\begin{lem}
We have 
\begin{equation}\label{normApprox}
    \widetilde{\mathbb P}\left(\max_{i\in S} \frac{S_{i,1}}{\sqrt{x_i}}\leq u \right)\leq \mathbb P\left(\max_{1\leq i\leq k} Z_i\leq u + \eta\right) + O\left(\frac{\log X}{\eta^3 X^{1/2}}\right), 
\end{equation}
where $\eta>0$ is small and the $Z_i$ are jointly normal random variables with the same (conditional) means and covariances as the (appropriately normalized) $S_{i,1}$. $\widetilde{\mathbb P}$ denotes conditional probability, conditioned on the values of $f(q)$ for all the primes $q\not\in\mathcal A$.  
\end{lem}

\begin{proof} This follows from Lemma \ref{normApproxLemma} applied to the first sums $S_{i,1}$ with $i\in S$, upon bounding the error term. The random variables $V_i$ are the $f(p)$ for $p\in\mathcal A$, together with the corresponding coefficients coming from \eqref{largeVal1} (after pulling out the $f(p)$). 

Due to $\mathcal A_i\cap \mathcal A_j = \emptyset$ unless $i=j$, the error terms simplify greatly to $$\frac{1}{\eta^{2} } \sum_{i\in S} \sqrt{ \sum_{p\in \mathcal A_i}  \E \left\lvert\frac{1}{\sqrt{x_i}} \sum_{\substack{n\leq x_i \\ p\mid n^2 + 1 \\ q\mid (n^2 + 1)/p \implies q\not\in\mathcal A}} f(n^2+1)\right\rvert^{4} }  + \frac{1}{\eta^{3}}
    \sum_{i\in S}
    \sum_{p \in  \mathcal A_i} \E  \left\lvert \frac{1}{\sqrt{
    x_i}} \sum_{\substack{n\leq x_i \\ p\mid n^2 + 1 \\ q\mid (n^2 + 1)/p \implies q\not\in\mathcal A}} f(n^2+1) \right\rvert^{3}.$$ Now simply using the point-wise bound $\ll 1$ on the sums over $n$ (since given $p\in\mathcal A_i$, there are at most $2$ values of $n\leq x_i$ such that $p\mid n^2 + 1$) gives that this is 
    \begin{equation}
        \ll \frac{1}{\eta ^2} \sum_{i\in S}\sum_{p\in\mathcal A_i} \frac{1}{x_i} + \frac{1}{\eta^3}\sum_{i\in S}\sum_{p\in\mathcal A_i} \frac{1}{x_i^{3/2}} \ll \frac{\lvert S\rvert}{\eta^2 X} + \frac{\lvert S\rvert }{\eta^3 X^{1/2}} \ll \frac{\log X}{\eta^3 X^{1/2}}.
    \end{equation}
In the last inequality we used the bound $\lvert \mathcal A_i\rvert\ll x_i$ (recall Lemma \ref{AsetProperties}) and that $\lvert S\rvert\leq\log X$.  
\end{proof}

We now apply Lemma \ref{normComparisonLemma} to create large fluctuations for at least one of the first sums $S_{i,1}$. First, we calculate the expected variance of $Z_i$, which is equal to the expectation of $\mathbb EZ_i^2$ over all realizations of $f(q)$ for $q\not\in\mathcal A$. We denote conditional expectation by $\widetilde{\mathbb E}$, conditioned on the values of $f(q)$ for $q\not\in\mathcal A$, and recall that $\mathbb EZ_i^2$ is equal to the (appropriately normalized) conditional variance of $S_{i, 1}$.  We have
\begin{align}
   \beta_i := \mathbb E\mathbb EZ_i^2  & = \frac{1}{x_i}\mathbb E\widetilde{\mathbb E} S_{i,1}^2 = \frac{1}{x_i}\mathbb E \sum_{\substack{n, m\leq x_i \\ p\mid \gcd(n^2 + 1, m^2 + 1) \text{ for some } p\in \mathcal A_i \\ q\mid (n^2 + 1)/p\implies q\not\in\mathcal A \\ q\mid (m^2 + 1)/p\implies q\not\in\mathcal A}} f(n^2+1)f(m^2 + 1) \\ & = \frac{1}{x_i}\sum_{\substack{n\leq x_i \\ p\mid n^2 + 1 \text{ for some } p\in\mathcal A_i \\ q\mid (n^2 + 1)/p \implies q\not\in \mathcal A}} \mu^2(n^2 + 1)
\gg 1 \end{align} by Lemma \ref{goodChunk} (we used $\mathbb Ef(m^2 + 1)f(n^2 + 1) \neq 0$ (and in fact equals $1$) if and only if $m=n$). 

Further, we calculate the variance of the variance of $Z_i$ over all realizations of $f(q)$ for $q\not\in\mathcal A$. We have 
\begin{align}
    \mathbb E(\mathbb EZ_i^2)^2 - \beta_i^2 & =  \frac{1}{x_i^2} \mathbb E\left(\sum_{\substack{n, m\leq x_i \\ p\mid \gcd(n^2 + 1, m^2 + 1) \text{ for some } p\in \mathcal A_i \\ q\mid (n^2 + 1)/p\implies q\not\in\mathcal A \\ q\mid (m^2 + 1)/p\implies q\not\in\mathcal A}} f(n^2+1)f(m^2 + 1)\right)^2 - \beta_i^2\\ & = \frac{1}{x_i^2} \mathbb E \sum_{\substack{n_1, m_1, n_2, m_2\leq x_i \\ P^+(n_j^2 + 1) = P^+(m_j^2 + 1) =: p_j\in\mathcal A_i \,\forall j=1, 2\\ q\mid (n_j^2 + 1)/p_j \implies q\not\in\mathcal A \\ q\mid (m_j^2 + 1)/p_j \implies q\not\in\mathcal A}} f(m_1^2 + 1)f(n_1^2 + 1)f(m_2^2 + 1)f(n_2^2 + 1)  - \beta_i^2 \\ & = \frac{1}{x_i^2} \#\left\{m_j, n_j\leq x_i : P^+(n_j^2 + 1) = P^+(m_j^2 + 1) \,\forall j=1, 2, \prod_{j=1}^2 (m_j^2+1)(n_j^2+1) = \square\right\}  - \beta_i^2. \label{varOfVar}
\end{align}
In \eqref{varOfVar} we require that $m_j^2 + 1$ and $n_j^2 + 1$ are square-free for $j=1, 2$. If $P^+(n_1^2+1) = P^+(n_2^2+1)$ (which forces all of the greatest prime factors to be equal), the number of such $m_j, n_j \, (j=1, 2)$ is $\ll 1$. Otherwise, by Proposition \ref{counting_squares}, we have a power-saving bound for \eqref{varOfVar} which counts non-diagonal solutions to $\prod_{j=1}^2 (m_j^2+1)(n_j^2 + 1) = \square$. It follows by Markov's inequality and the union bound over all $i\in S$ that with probability $1-O(X^{-\delta})$ we have $\min_{i\in S} \mathbb EZ_i^2 \geq m$ for some $m>0$. Now we have 
\begin{equation}
    \mathbb P\left(\max_{i\in S} Z_i\leq u + \eta \right) = \mathbb P\left(\max_{i\in S} \frac{Z_i}{\sqrt{\mathbb EZ_i^2}}\leq \frac{u + \eta}{\sqrt{\min_{i\in S}\mathbb EZ_i^2}} \right).
\end{equation}
Applying Lemma \ref{normComparisonLemma} (note that $\mathbb EZ_iZ_j = 0$ since $\mathcal A_i\cap \mathcal A_j = \emptyset$ for $i\neq j$) with $u=\sqrt{m\log\log X}$, $\delta = 1/100$, $\varepsilon = 1/X$ and $\eta$ a fixed constant, we get 
\begin{equation}
    \mathbb P\left(\max_{i\in S} \frac{Z_i}{\sqrt{\mathbb EZ_i^2}}\leq \frac{u + \eta}{\sqrt{\mathbb EZ_i^2}} \right)\leq \mathbb P\left(\max_{i\in S} \frac{Z_i}{\sqrt{\mathbb EZ_i^2}}\leq \frac{\sqrt{m\log\log X} + \eta}{\sqrt{m}} \right)\leq e^{-\Theta((\log X)^{1/3000})}. \label{finalEq}
\end{equation}

It follows by \eqref{normApprox} that the maximum of the $S_{i,1}$ for $i\in S$ gets larger than $\sqrt{x_i\log\log x_i}$ with probability $1- e^{-\Theta((\log X)^{1/3000})}$. Combining this with the existence of a suitable set $S$ with probability $1-O((\log\log X)^{-1/50})$ over the realizations of $f(q)$ for $q\not\in\mathcal A$ (recall the conclusion of subsection \ref{2nd3rdsums}) concludes the proof of Lemma \ref{large_values_lemma}.

\section{Acknowledgements}

The authors thank: Jonathan Bober and Oleksiy Klurman for many helpful conversations throughout the research phase of this project and for their comments on the final draft of this paper; Tim Browning for a fruitful discussion leading to an improvement in Lemma \ref{New_BP}; Matthew Bisatt, Andrew Granville, Ross Paterson, and Per Salberger for helpful discussions concerning integral points on families of varieties; and the anonymous referee for their diligence in reviewing this manuscript. Much of this work was supported by the Swedish Research Council, grant no. 2021-06594, while the authors were in residence at Institut Mittag-Leffler in Djursholm, Sweden, during the winter semester of 2024; the authors thank the institute and, in particular, the organizers for excellent working conditions. The first author is supported by an NSERC postdoctoral fellowship, application no. PDF – 567986 – 2022; the second author is funded by a University of Bristol PhD scholarship.

\bibliographystyle{alpha}
\bibliography{ref}

\textsc{School of Mathematics, University of Bristol, Woodland Road, Bristol BS8 1UG, UK}

\textit{Email address:} \href{mailto:jake.chinis@bristol.ac.uk}{{\texttt{jake.chinis@bristol.ac.uk}}}\\
\indent\textit{Email address:} \href{mailto:besfort.shala@bristol.ac.uk}{{\texttt{besfort.shala@bristol.ac.uk}}}

\end{document}